\documentclass[12pt,reqno,oneside,a4paper]{article}

\usepackage{amsthm}
\usepackage{amsmath}
\usepackage{float}
\usepackage{graphicx}
\usepackage{enumitem}
\usepackage{titlesec}
\usepackage{enumitem}
\usepackage{amssymb}
\usepackage[nottoc,notlot,notlof]{tocbibind}
\usepackage{framed}
\usepackage{blindtext}
\usepackage{mathrsfs}
\usepackage{realboxes}
\usepackage{chngcntr}
\usepackage[utf8]{inputenc}
\usepackage{relsize}
\usepackage{mathtools}
\usepackage{subcaption}
\usepackage{pgf,tikz,pgfplots}
\usepackage{mathrsfs}
\usetikzlibrary{arrows}
\newtheorem{theorem}{Theorem}
\newtheorem{lemma}[theorem]{Lemma}
\newtheorem{cor}[theorem]{Corollary}
\newtheorem{prop}[theorem]{Proposition}

\theoremstyle{definition}
\newtheorem{definition}[theorem]{Definition}

\theoremstyle{remark}
\newtheorem*{remark}{Remark}

\newcommand{\pra}{\mathbb{R}}
\newcommand{\mig}{\mathbb{C}}
\newcommand{\fis}{\mathbb{N}}

\makeatletter
\newcommand{\displaybump}{\hbox to \@totalleftmargin{\hfil}}
\makeatother
\numberwithin{theorem}{section} 
\numberwithin{equation}{section}

\title{Cantor bouquets in spiders' webs}
\author{Yannis Dourekas}
 \date{}
\begin{document}

\maketitle

\begin{abstract}
For many transcendental entire functions, the escaping set has the structure of a Cantor bouquet, consisting of uncountably many disjoint curves. Rippon and Stallard showed that there are many functions for which the escaping set has a new connected structure known as an infinite spider’s web. We investigate a connection between these two topological structures for a certain class of sums of exponentials.
\end{abstract}

\section{Introduction}
        Let $f : \mig \to \mig$ be a transcendental entire function. We denote by $f^n$ the $n$th iterate of $f$, for $n=0, 1, 2 \ldots$. The \emph{Fatou set} of $f$, $F(f)$, is the set of points $z \in \mig$ such that the sequence $\{f^n\}_{n \in \fis}$ forms a normal family in some neighborhood of $z$. The complement of the Fatou set is the \emph{Julia set} of $f$, $J(f)$. Another set of note is the \emph{escaping set} of $f$, $I(f)$, defined as the set of points that tend to infinity under iteration. Further, we can define the \emph{fast escaping set} of $f$, $A(f)$, roughly defined as the set of points that tend to infinity under iteration ``as fast as possible''. The formal definition of the fast escaping set, which can be found in \cite{rs1}, along with an extensive study of many of its properties, is
\begin{align*}
A(f) = \{z \in \mig: \exists L \in \fis \text{ such that } |f^{n+L}(z)| \geq M^n(R,f) \text{ for } n \in \fis \},
\end{align*}
where
\begin{align*}
M(r,f) = \max_{|z|=r} |f(z)|, \text{ for } r>0,
\end{align*}
$M^n(r,f)$ denotes iteration of $M(r,f)$ with respect to the variable $r$, and $R>0$ is any value large enough so that $M(r,f)>r$ for $r \geq R$.

In the same paper, the notion of a \emph{spider's web} is introduced. This is a connected structure containing a nested sequence of loops. The formal definition is as follows:

\begin{definition}
A set $E$ is an (infinite) spider's web if $E$ is connected and there exists a sequence of bounded simply connected domains $G_n$, with $G_n \subset G_{n+1}$, for $n \in \fis$, $\partial G_n \subset E$, for $n \in \fis$ and $\cup_{n \in \fis} G_n = \mig$.
\end{definition}
It is known that the escaping, fast escaping, and even Julia sets of many transcendental entire functions are spiders' webs  \cite{rs1}. We note that the spiders' webs that arise in complex dynamics are extremely elaborate \cite{john, rs1}. 

On the other hand, the escaping set for certain families of exponential functions has been found to consist of a set of uncountably many, pairwise disjoint curves, referred to as a \emph{Cantor bouquet}. Results on this can be found for the family of functions $\lambda e^z$ for  $0<\lambda < 1/e$ \cite{dev1, dev2} and $\lambda \in \mig$ \cite{dev3}. This structure has been closely associated to symbolic dynamics that arise from the dynamical properties of these functions, as can be seen in the aforementioned papers as well as \cite{lx, sz}. Informally, a Cantor bouquet can be thought of as the Cartesian product of a Cantor set with $[0, +\infty)$ \cite[p. 490]{dev1}. A topological description of a Cantor bouquet can be found in \cite{aarts}, with additional discussion provided in \cite{lx}.

In this paper we prove that it is possible for a Cantor bouquet (or indeed many Cantor bouquets) to lie within an escaping or Julia set spider's web.

To that end, we study the family of transcendental entire functions defined as
\begin{align*}
\mathcal{F} = \left\{ f: f(z) = \sum_{k=0}^{p-1} \exp( \omega_p^k z) \text{ for any } p \geq 3 \right\},
\end{align*}
where $\omega_p= \exp(2 \pi i/p)$ is a $p$th root of unity. A larger class of functions that includes this family was studied in \cite{dave2}. One of the main tools used in this study was the fact that, for each $p \geq 3$, there exist $p$ unbounded regions outside a circle centered at the origin with the property that $f$ behaves like a single exponential in each one of them. Each of these regions is a $2 k \pi/p$ rotation of the others, for $k=1, \ldots, p-1$ (see the lemma in the next section).

It was shown in \cite{dave2} that the Julia, escaping and fast escaping sets of each of these functions is a spider's web. Our goal is to prove that, within each of these spiders' webs there lie Cantor bouquets. We formulate our main result as follows:

\begin{theorem}
Let $f \in \mathcal{F}$. Then $J(f)$ is a spider's web that contains a Cantor bouquet. Additionally, the curves minus the endpoints lie in $A(f)$.
\end{theorem}

\begin{remark}
The result holds for functions of the form $\lambda f$, where $f \in \mathcal{F}$ and $\lambda > 0$. The reasoning is similar but for simplicity we have given the proof for the case $\lambda=1$. Additionally, when $p$ is even, $f$ is an even function, and the result holds for negative $\lambda$ as well.
\end{remark}

We will prove our result for just one of the $p$ regions mentioned above; due to symmetry, the curves we find in that region will have $2k \pi/p$ rotations in all the other regions for $k=1, \ldots, p-1$, and these rotations will have similar properties. Throughout the rest of the paper $p$ is a fixed integer with $p \geq 3$.

The structure of the proof is as follows:
\begin{itemize}
\item In Section \ref{sec:pre} we list the preliminary results from \cite{dave2} that we will make use of.
\item In Section \ref{sec:zeros} we find all the zeros and critical points for all $f \in \mathcal{F}$.
\item In Section \ref{sec:trapeziums} we find different subsets of the plane that cover themselves under $f$: the endpoints of the curves in the Cantor bouquet will lie in these regions.
\item In Section \ref{sec:hairs} we identify curves that extend to infinity (which we call \emph{hairs}), prove that they do, in fact,  constitute a Cantor bouquet, and further prove that they are in $J(f)$ and $A(f)$ (apart from the endpoints), thus making them part of the $J(f)$, $A(f)$, and $A(f) \cap J(f)$ spiders' webs.
\end{itemize}
Our argument in Sections \ref{sec:trapeziums} and \ref{sec:hairs} is inspired by \cite{dev3}, where the authors prove the existence of hairs in the dynamical plane for the family of complex exponential functions $z \mapsto \lambda e^z$ for $\lambda \in \mig$. In our case, the functions in $\mathcal{F}$ provide extra challenges (for example, locating the critical points and finding regions that cover themselves under iteration), since they arise as sums of exponentials and, further, are not in the Eremenko--Lyubich class. Several different tools, including Laguerre's theorem, are thus required in order to prove our results.\\

\emph{Acknowledgements.} I would like to thank my supervisors,  Prof Phil Rippon and Prof Gwyneth Stallard, for their boundless patience and unsparing guidance.

\section{Preliminaries}\label{sec:pre}
This section contains some of the preliminary results we will use, many taken from \cite{dave2} with slight modifications in order to make their purpose clearer for the requirements of this paper. We also prove a lemma that follows as a corollary from these results.

We first prove a result that concerns the symmetry properties of $\mathcal{F}$. This will allow us to work in certain angles in order to prove results for the whole plane, and justifies the use of the phrase ``due to symmetry'' that will appear numerous times in what follows.
\begin{lemma} \label{symmetry}
Let $f \in \mathcal{F}$. Then $f( \omega_p^k z) = f(z)$ for $k=1, \ldots, p-1$ and for all $z \in \mig$.
\end{lemma}
\begin{proof}
It suffices to prove that $f( \omega_p z) = f(z)$. We have
\begin{align*}
 f(z) = \sum_{k=0}^{p-1} \exp( \omega_p^k z),
\end{align*}
so
\begin{align*}
 f( \omega_p z) &= \sum_{k=0}^{p-1} \exp( \omega_p^k  \omega_p z)
=  \sum_{k=0}^{p-1} \exp( \omega_p^{k+1}   z)
\\&=  \sum_{k=1}^{p-1} \exp( \omega_p^k   z) +  \exp( \omega_p^{p}   z)
=  \sum_{k=1}^{p-1} \exp( \omega_p^k   z) +  \exp( z)
\\&=   \sum_{k=0}^{p-1} \exp( \omega_p^k z)
=   f(z).\tag*{\qedhere} 
\end{align*}
\end{proof}

The following result is one we have already mentioned in the previous section \cite[Theorem 1.2]{dave2}:

\begin{theorem}
Suppose that $f \in \mathcal{F}$. Then each of 
\begin{align*}
A(f), I(f), J(f) \cap A(f), J(f)\cap I(f), \text{ and } J(f)
\end{align*}
is a spider's web.
\end{theorem}

We will now fix several constants that serve as preparation for the next lemma, which concerns the fact that there are certain unbounded regions in which $f$  behaves like a single exponential.

Choose a constant $\sigma$ such that
\begin{align*}
0 < \sigma < \frac{1}{8 \sqrt{2}}.
\end{align*}
Fix a constant $\eta > 4 / \sigma$. Fix also a constant $\tau$ sufficiently large that
\begin{align*}
 \tau \geq \frac{\log (4 p \eta) }{2 \sin (\pi/p)} >0.
\end{align*}
Suppose that $\nu > 0$ is large compared to $\tau$; we will specify its size more precisely below.
 Let $P(\nu)$ be the interior of the regular $p$-gon centered at the origin and with vertices at the points
\begin{align*}
\frac{\nu}{\cos(\pi/p)} \exp \left( \frac{(2k+1)i \pi}{p} \right), \text{ for } k \in \{0,1,\cdots,p-1\}.
\end{align*}
Define the domains
\begin{align*}
Q_k = \left\{ z \exp \left( \frac{(2k+1)i \pi}{p} \right) : \operatorname{Re} (z) > 0, |\operatorname{Im}(z)| <  \tau \right\}, \text{ for } k \in \{0,1,\cdots,p-1\}.
\end{align*}
Roughly speaking, each $Q_k$ can be obtained by rotating a half-infinite horizontal strip of width $2 \tau$ around the origin until a vertex of $P(\nu)$ is positioned centrally in the strip.

Set
\begin{align*}
T( \nu) = \mig \setminus \left( P(\nu) \cup \bigcup_{k=0}^{p-1} Q_k \right).
\end{align*}

The set $T(\nu)$ consists of $p$ simply connected unbounded components.  These components are arranged rotationally symmetrically. We label these $T_j(\nu)$, for $j \in  \{0,1,\cdots,p-1\}$, where $T_0(\nu)$ is the component that has an unbounded intersection with the positive real axis. Then, each component of $T_{j+1} (\nu)$ is obtained by rotating $T_j(\nu)$ clockwise around the origin by $2 \pi /p$ radians; see Figure 1. We take $\nu>0$ so large that
\begin{equation}
|e^z| \geq 4 p \eta | \exp(\omega^k_p z)| \label{eq:egine}
\end{equation}
for $k = 1, 2, \ldots p-1$ and $z \in T_0(\nu)$.
The following lemma shows that for such $\nu$, $f$ behaves like a single exponential in each component of $T(\nu)$. Lemma \ref{davel} is a special case of \cite[Lemma 4.1]{dave2}.

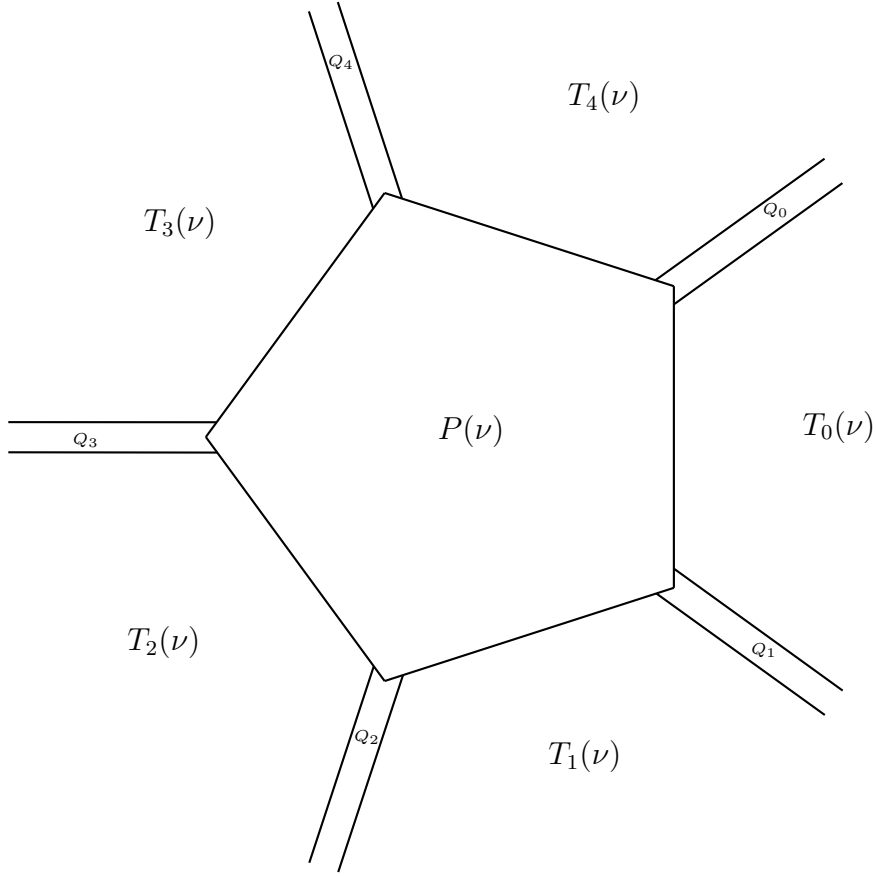
\begin{figure}[h]
\centering
\begin{subfigure}[b]{.8\textwidth}
\begin{tikzpicture}
\draw [line width=0.8pt] (2.75316535371619,-1.9972747496061851)-- (2.75316535371619,2.0027252503938153);
\draw [line width=0.8pt] (2.75316535371619,2.0027252503938153)-- (-1.0510607114644228,3.2387932278936047);
\draw [line width=0.8pt] (-1.0510607114644228,3.2387932278936047)-- (-3.402201720634316,0.002725250393815548);
\draw [line width=0.8pt] (-3.402201720634316,0.002725250393815548)-- (-1.0510607114644244,-3.233342727105974);
\draw [line width=0.8pt] (-1.0510607114644244,-3.233342727105974)-- (2.75316535371619,-1.9972747496061851);
\draw [line width=0.8pt] (-6,0.2)-- (-3.2600367285857894,0.1983985751733357);
\draw [line width=0.8pt] (-3.252230792018594,-0.2036920243491368)-- (-6,-0.2);
\draw [line width=0.8pt] (-2.0417214023631414,-5.642652597187732)-- (-1.1935031420162177,-3.0372875409865143);
\draw [line width=0.8pt] (-0.808680090143069,-3.1541164826793517)-- (-1.6612987958450798,-5.766259394937711);
\draw [line width=0.8pt] (4.738146777781479,-3.683584888353489)-- (2.5224112211399743,-2.071779305693192);
\draw [line width=0.8pt] (2.752439010284849,-1.7418929630065727)-- (4.973260878698468,-3.3599780906035104);
\draw [line width=0.8pt] (4.970057154717891,3.3698381391561885)-- (2.7524390102848493,1.7606237162958618);
\draw [line width=0.8pt] (2.509780950460227,2.0813336301934253)-- (4.734943053800901,3.6934449369061664);
\draw [line width=0.8pt] (-1.6664825301362274,5.770025598354111)-- (-0.8213103608228155,3.163670807179586);
\draw [line width=0.8pt] (-1.2013090785834117,3.031994091810713)-- (-2.0469051366542876,5.646418800604131);
\draw (-0.49590073949914104,0.43137351827695) node[anchor=north west] {$P(\nu)$};
\draw (3.7789459447540286,3.2517304991653325) node[anchor=north west] {\tiny{$Q_0$}};
\draw (-1.9394925021909722,5.217118143203081) node[anchor=north west] {\tiny{$Q_4$}};
\draw (-5.295387698637927,0.19820006277223147) node[anchor=north west] {\tiny{$Q_3$}};
\draw (-1.59646124105624335,-3.740628772436522) node[anchor=north west] {\tiny{$Q_2$}};
\draw (3.6234969744175496,-2.5832946755320314) node[anchor=north west] {\tiny{$Q_1$}};
\draw (4.303586219639644,0.4953556692325308) node[anchor=north west] {$T_0(\nu)$};
\draw (0.9614333574053485,-3.876646621480941) node[anchor=north west] {$T_1(\nu)$};
\draw (1.2140379342021268,4.867357959946003) node[anchor=north west] {$T_4(\nu)$};
\draw (-4.362693876619053,3.1962815288288535) node[anchor=north west] {$T_3(\nu)$};
\draw (-4.576436210831712,-2.361019160700271) node[anchor=north west] {$T_2(\nu)$};
\end{tikzpicture}
\end{subfigure}\caption{The sets $P(\nu)$, $T_k(\nu)$ and $Q_k$ for $p=5$ and $k \in \{0, 1, 2, 3, 4\}$.}
\end{figure}

\begin{lemma} \label{davel}
Let $f \in \mathcal{F}$. Suppose that $\eta, \tau, T_j(\nu)$ and $T(\nu)$ are as defined above, for $j \in \{0,1, \cdots, p-1\}$. Then there exists $\nu'>0$ such that the following holds. Suppose that $\nu \geq \nu'$; there exists a constant $\varepsilon_0 \in (0,1)$, independent of $\nu$, such that, for all $z \in T(\nu)$,
\begin{equation}
|f'(z)|>2  \label{eq:daveleq1}
\end{equation}
\begin{equation}
\left| z \frac{f'(z)}{f(z)} \right| >2,\label{eq:daveleq2}
\end{equation}
and finally
\begin{equation}
|f(z)| > \max\{ e^{\varepsilon_0 \nu}, M ( \varepsilon_0 |z|, f)\}.\label{eq:daveleq3}
\end{equation}
\end{lemma}
For the rest of this paper we will be working in $T_0(\nu)$. Similar arguments work for the other components of $T(\nu)$ due to symmetry.

We fix the notation found in \cite{dev3} that we also use here.
For each integer~$k$, we define horizontal strips $R(k)$ by
\begin{align*}
R(k) = \{ z \in \mig: (2k-1) \pi < \operatorname{Im} z < (2k+1)\pi \}.
\end{align*}
 Note that $z \mapsto e^z$ maps the boundary of $R(k)$ onto the negative real axis and $z \mapsto e^z$ maps $R(k)$ onto $\mig \setminus \{ x \in \pra: x \leq 0\}$ for each integer $k$. 

\begin{definition}
For $z \in \mig$, the \emph{itinerary} of $z$ under $f$ is the sequence of integers $s(z) = s_0 s_1 s_2 \ldots$ where $s_n = k$ if and only if $f^n(z) \in R(k)$. We do not define the itinerary of $z$ if $f^n(z) \in \cup_{k \in \fis} \partial R(k)$ for some $n$.
\end{definition}

Let $N \in \fis$. Let $\Sigma_N$ consist of all one-sided sequences $s_0 s_1 s_2 \ldots$, where each $s_j \in \mathbb{Z}$ and $|s_j| \leq N$. The one-sided shift $\sigma$ on $\Sigma_N$ is defined by
\begin{align*}
\sigma (s_0 s_1 s_2 \ldots) = s_1 s_2 s_3 \ldots.
\end{align*}
It is known that $\sigma$ has dense periodic points in $\Sigma_N$, has dense orbits, and exhibits sensitive dependence on initial conditions (see for example \cite[Chapter 3]{dyn}).

We finish this section by proving the following lemma for points that stay in $T_j(\nu)$, $j \in \{0, 1, \ldots, p-1\}$, under iteration. 

\begin{lemma}\label{g} Let $z \in \mig$ be such that $f^n(z) \in T_j(\nu)$ for all $n \geq 1$ and $j \in \{0, 1, \ldots, p-1\}$. If $\nu$ is large enough, then $z \in J(f) \cap A(f)$.
\end{lemma}

To prove this we need the following: First, a lemma from \cite{dave2}.

\begin{lemma}\label{davel2} Suppose that $f$ is a transcendental entire function and that $z_0 \in I(f)$. Set $z_n=f^n(z_0)$, for $n \in \fis$. Suppose that there exist $\lambda>1$ and $N \geq 0$ such that
\begin{align*}
f(z_n) \neq 0 \quad \text{and} \quad \left| z_n \frac{f'(z_n)}{f(z_n)} \right| \geq \lambda, \quad \text{for } n \geq N.
\end{align*}
Then either $z_0$ is in a multiply connected Fatou component of $f$ or $z_0 \in J(f)$.
\end{lemma}
Second, we state a corollary from \cite{dave2} (proved using \cite[Theorem 4.5]{berg}).
\begin{cor} Suppose $f \in \mathcal{F}$. Then $f$ has no multiply connected Fatou components. \label{coro}
\end{cor}

\begin{proof}[Proof of Lemma \ref{g}]
Let $z \in \mig$ be such that $f^n(z) \in T_j (\nu)$ for all $n \geq 0$ and $j \in \{0, 1, \ldots, p-1\}$. If $\nu$ is large enough, it follows from  \eqref{eq:daveleq3} that $z \in I(f)$. 
We will prove that $z \in A(f)$. Let $0<\varepsilon_0<1$ be the constant from Lemma \ref{davel}. 
We need a standard result about $M(r)$ (see, for example,  \cite[p. 7, (2.3)]{rs1}): if $k>1$, then
\begin{align*}
\frac{M(kr)}{M(r)} \to \infty \text{ as } r \to \infty.
\end{align*}
Therefore, taking $k = 1/ \varepsilon_0$, we have
\begin{equation}
M (r) = M \left( \frac{ \varepsilon_0 r}{\varepsilon_0} \right) \geq \frac{1}{\varepsilon_0^2} M( \varepsilon_0 r) \label{eq:e0}
\end{equation}
for all large enough $r>0$, say $r \geq \varepsilon_0 r_0 > 0$. Further, from \eqref{eq:daveleq3}, we have
\begin{align*}
|f(z)| \geq M( \varepsilon_0 |z|)
\end{align*}
since $z \in T_j(\nu)$ for some $j \in \{0, 1, \ldots, p-1\}$, and thus, by $\eqref{eq:e0}$ with $|z| \geq \varepsilon_0 r_0$,
\begin{align*}
|f(z)| \geq  \frac{1}{\varepsilon_0^2} M( \varepsilon_0^2 |z|).
\end{align*}
Additionally, by substitution and the previous inequality,
\begin{align*}
|f^2(z)| \geq  \frac{1}{\varepsilon_0^2} M( \varepsilon_0^2 |f(z)|) \geq  \frac{1}{\varepsilon_0^2} M (M( \varepsilon_0^2 |z|)),
\end{align*}
and thus, using induction, we have
\begin{align*}
|f^n(z)| \geq  \frac{1}{\varepsilon_0^2} M^n( \varepsilon_0^2 |z|) \geq  M^n( \varepsilon_0^2 |z|) 
\end{align*}
for all $n \geq 0$ and all large enough $|z|$. Therefore, for these $z$ with large enough moduli, $z \in A(f)$.
But the other $z \in \mig$ for which $f^n(z) \in T_j (\nu)$  for all $n \geq 0$ and $j \in \{0, 1, \ldots, p-1\}$ are in $I(f)$, so their moduli will get as large as we want, making them preimages of points in $A(f)$. Consequently, they are in $A(f)$ as well.

We now prove that $z \in J(f)$. From \eqref{eq:daveleq2} we have  
\begin{align*}
\left| f^n(z) \frac{f'(f^n(z))}{f(f^n(z))} \right| >2
\end{align*} 
for $n \geq 0$. Since $z \in I(f)$, it follows from Lemma \ref{davel2} that either $z \in J(f)$ or $z$ is in a multiply connected Fatou component. The latter case is impossible by Corollary \ref{coro}, so $z \in J(f)$. 
\end{proof}

\section{Zeros and critical points}\label{sec:zeros}

Recall that $f(z) = \sum_{k=0}^{p-1}\exp \left( \omega_p^k z \right)$ for some fixed $p \geq 3$. In this section we will locate the zeros of $f$. These will, in turn, lead us to the location of the critical points and the critical values of $f$, and later on allow us to locate bounded sets that cover themselves under iteration. 

We claim that all the zeros of $f$ lie on the rays $V_0, \ldots, V_{p-1}$, where $V_0:=\{x+iy \in \mig: y = \tan(\pi/p) x,  x>0\}$ and $V_1, \ldots, V_{p-1}$ are its $2 k \pi/p$-rotations around the origin for $k=1, \ldots, p-1$ respectively. The main tool used to prove this is the following result which we quote as a lemma \cite[Problem 160]{polya}:
\begin{lemma} \label{polyal}
Let $q$ be an integer, $q \geq 2$. The entire function
\begin{align*}
F(z) = 1 + \frac{z}{q!} + \frac{z^2}{(2q!)}+ \frac{z^3}{(3q!)}  + \ldots
\end{align*}
has no non-real zeros.
\end{lemma}

We now state and prove our result about the zeros of $f$.

\begin{theorem}
Let $f \in \mathcal{F}$. Then all the zeros of $f$ lie on the rays $V_0, \ldots, V_{p-1}$. \label{3.2}
\end{theorem}
\begin{proof}
We write
\begin{align*}
f(z) &=  \sum_{k=0}^{p-1}\exp \left( \omega_p^k z \right) = \sum_{k=0}^{p-1} \sum_{j=0}^{\infty} \left( \frac{(\omega_p^k z)^j}{j!} \right)
= p \left( 1 + \frac{z^p}{p!} + \frac{z^{2p}}{(2p)!} + \ldots \right),
\end{align*} 
since $1+ \omega_p^k +  (\omega_p^k )^2 + \ldots + (\omega_p^k)^{p-1} = 0$ for $k = 1, \ldots, p-1$.

We substitute $z=w^{1/p}$, the principal branch, to obtain
\begin{equation} \label{eq:below}
g(w) = f(w^{1/p}) = p \left( 1 + \frac{w}{p!} + \frac{w^2}{(2p)!} + \ldots \right).
\end{equation}
We can apply Lemma \ref{polyal} to the function $g$ to deduce that all the zeros of $g$ are real. Hence, the zeros of $f$ must lie on the preimages of the real axis under $z \mapsto z^{1/p}$. These are exactly the rays $V_0, \ldots, V_{p-1}$, which are the preimages of the negative real axis, along with the positive real axis, which is the preimage of itself. But it is simple to see that $f(x) \neq 0$ for $x\geq 0$. 
\end{proof}

In fact, we can say more about the location of the zeros. In particular, we can describe their distribution in a neighbourhood of infinity.

First, we introduce some notation and establish some symmetry properties of $f$. Consider one of the terms of the sum defining $f$; its general form is 
\begin{align*}
\exp( \omega_p^k z) &= \exp \left( e^{2 i k \pi/p} (x+iy) \right) = \exp (u_k(z)) \exp(i v_k(z)),
\end{align*}
with $z=x+iy$,
\begin{equation}
u_k(z) =  x \cos \left( \frac{ 2k \pi}{p} \right) - y \sin  \left( \frac{ 2k \pi}{p} \right)  \label{eq:gk}
\end{equation}
and
\begin{equation}
v_k(z) =  x  \sin  \left( \frac{ 2k \pi}{p} \right) + y  \cos \left( \frac{ 2k \pi}{p} \right) . \label{eq:hk}
\end{equation}
We also define $v_k(z)$ to be equal to $\pi/3$ for $k =p/2-1$ when $p$ is odd, to make the statement of the next lemma more concise (the reason for this will appear near the end of the proof). We prove our results for $V_0$; analogous results follow for the rest of the rays due to symmetry. We start by proving a lemma that simplifies the equation of $f$ on $V_0$, in particular showing that $f$ is real on $V_0$, and so $f$ is real on each of the rays $V_k$, $k \in \{1, \ldots, p-1\}$, by symmetry; this fact also follows from \eqref{eq:below}, since if $z \in V_k$, then $z^p \in \pra$ and $f(z) = g(z^p)$.
\begin{lemma} \label{arguments}
Let $z \in V_0$. We have
\begin{align*}
f(z) = 2 \sum_{k =0, 1 , \ldots, p/2-1} \exp (u_k(z)) \cos (v_k(z)). 
\end{align*}
\end{lemma}
\begin{proof}
We show that the terms of the sum defining $f$ act similarly in pairs on the ray $V_0$ with regards to their moduli, which are proven to be equal for specific pairs, as well as their arguments which, for the same pairs, are proven to be of opposite sign. In particular, the $k$-th term, for $k<(p-1)/2$, behaves similarly to the $(p-k-1)$-th term. 

In particular, a point in $V_0$ is of the form $r \exp( \pi i/ p)$ for $r>0$. By substitution we get
\begin{align*}
u_k(r \exp( \pi i/ p)) = \operatorname{Re} \exp ( e^{2 \pi i k/p} re^{ \pi i/p}) = e^r  \cos \left( \pi (2k+1) / p \right) 
\end{align*}
and
\begin{align*}
v_k(r \exp( \pi i/ p)) = \operatorname{Im} \exp ( e^{2 \pi i k/p} re^{ \pi i/p}) = e^r \sin  \left( \pi (2k+1) / p \right) .
\end{align*}
On the other hand,
\begin{align*}
u_{p-k-1}(r \exp( \pi i/ p)) &= \operatorname{Re} \exp ( e^{2 \pi i (p-k-1)/p} re^{ \pi i/p}) 
\\&= e^r  \cos \left( \pi (2p-2k-1)/p \right) 
\\&=  e^r  \cos \left( \pi (-2k-1)/p \right),
\end{align*}
so
\begin{equation}
u_k(r \exp( \pi i/ p)) =  u_{{p-k-1}}(r \exp( \pi i/ p)),  \label{eq:mods}
\end{equation}
and
\begin{align*}
v_k(r \exp( \pi i/ p)) &= \operatorname{Im} \exp ( e^{2 \pi i k/p} re^{ \pi i/p}) = e^r \sin \left( \pi (2p-2k-1)/p \right) 
\\&=  e^r  \sin \left( \pi (-2k-1)/p \right),
\end{align*}
so
\begin{equation}
v_k(r \exp( \pi i/ p)) = - v_{{p-k-1}}(r \exp( \pi i/ p)). \label{eq:args2}
\end{equation}
Therefore, for $z = r \exp( \pi i/ p)$, the following sum is real:
\begin{align*}
\exp( \omega_p^k z)  + \exp( \omega_p^{p-k-1} z) &= \exp (u_k(z)) \exp(i v_k(z)) + \exp (u_{p-k-1}(z)) \exp(i v_{p-k-1}(z)) 
\\&=  \exp (u_k(z)) \exp(i v_k(z)) +  \exp (u_k(z)) \exp(-i v_k(z))
\\&= 2 \exp (u_k(z)) \cos (v_k(z)).
\end{align*}
Note that, if $p$ is odd, say $p=2m+1$, then, by $\eqref{eq:args2}$,
\begin{align*}
v_m(r \exp(\pi i/p)) = v_{(2m+1)-m-1}(r \exp( \pi i/ p)) = - v_m(r \exp( \pi i/p)) 
\end{align*}
so
\begin{align*}
v_m(r \exp( \pi i/ p)) = 0
\end{align*}
and thus the only term of the sum that does not belong to a pair, exclusively attains real values on $V_0$. Further, in this case
\begin{align*}
u_m(r \exp( \pi i/ p)) &= r \cos \left( \frac{\pi (2m+1)}{2m+1}\right) = -r,
\end{align*}
so $\exp (u_m(r \exp( \pi i/ p)))=e^{-r}$. For points $z \in V_0$, we can therefore write
\begin{align*}
f(z) = 2 \sum_{k=0}^{p/2-1} \exp (u_k(z)) \cos (v_k(z)) 
\end{align*}
for even $p$, and
\begin{align*}
f(z) = 2 \sum_{k=0}^{p/2-1} \exp (u_k(z)) \cos (v_k(z)) +  \exp (u_{(p-1)/2}(z))
\end{align*}
for odd $p$. Recalling our convention that $v_k(z) = \pi/3$ for $k =p/2-1$ when $p$ is odd, we can finally write, for $z \in V_0$ and all $p \geq 3$,
\begin{align*}
f(z) = 2 \sum_{k =0, 1 , \ldots, p/2-1} \exp (u_k(z)) \cos (v_k(z)). \tag*{\qedhere} 
\end{align*}  \end{proof}
We state an elementary lemma about real exponentials which we will make use of below several times. 

\begin{lemma} \label{rexp}
For  $d >0$ and $a<1$, let $E_{d,a} : \pra^+ \to \pra$ with $E_{d,a}(x)= e^x - de^{ax}$. Then, for $x(1-a) > \log^{+}(a d)$, and as $x \to \infty$, $E_{d,a}(x)$ increases to infinity.
\end{lemma}
\begin{proof}
We have
\begin{align*}
E_{d,a}(x) = e^x-de^{ax} = e^{ax} ( e^{x(1-a)}-d) \to \infty
\end{align*}
as $x \to \infty$, and is increasing for $x(1-a) > \log^{+} (ad)$.
\end{proof}

We are following the notation and the partition of the plane as introduced in \cite[p. 9757]{dave2} and discussed in Section \ref{sec:pre}. From Lemma \ref{davel} we know that there exist no zeros of $f$ in $T_0(\nu)$ and all potential zeros of $f$ should therefore lie in $\cup_{k=0, \ldots ,p-1} Q_k$. Due to symmetry, it suffices to locate all zeros in $Q_0$; the rest will be $2 \pi/p$ rotations of these. We can write
\begin{align*}
Q_0 =  \{ z \in \mig: z = w +  t \text{ for } w \in V_0  \text{ and } |t| \leq  \tau/\cos(\pi/p) \}.
\end{align*}
We define the lines
\begin{equation}
C_m = \left\{ x+iy : y = - \cot(\pi/p)x + \frac{m \pi}{\sin^2 ( \pi/p)} \right\} \label{eq:cot}
\end{equation}
for $m \in \fis$. Note that $C_m$ meets $V_0$ at $(m \pi \cot (\pi/p), m \pi)$. As pointed out in \cite[p. 9766]{dave2}, it is easy to check that
\begin{equation}
\operatorname{arg}(e^z) = \operatorname{arg} (e^{\omega_p^{p-1} z}) \label{eq:newarg},
\end{equation}
for $z \in C_m$, $m \in \fis$.
This is important since, in the part of the plane near $Q_0$, these two terms are much larger in terms of their modulus than the rest of the terms that make up the sum that defines $f$. For all $m \in \fis$ we finally consider the rectangle defined by the lines $C_m$ and $C_{m+1}$ along with the rays $V_0 \pm (\tau/\cos(\pi/p))$, and name it $D_m$.

We now describe the distribution of the zeros of $f$.
\begin{theorem} Let $f(z) =  \sum_{k=0}^{p-1}\exp( \omega_p^k z)$ and consider the set of the rectangles $D_m \subset Q_0$ for $m \in \fis$, as well as all their rotations around the origin that lie in $Q_1, \ldots, Q_{p-1}$. There exists $M \in \fis$ such that for $m > M$, there is exactly one zero of $f$ inside each one of the $p$ rectangles corresponding under symmetry to $D_m$, which additionally lies on one of the rays $V_k$ for $k=1, \ldots, p-1$. These are the only zeros of $f$ with a modulus larger than $M / \sin(\pi/p)$. \label{zeros}
\end{theorem}

\begin{proof}
Again, we restrict our calculations to $Q_0$. We use Rouché's theorem to prove that, for large enough $M$, there is exactly one zero in each of the rectangles $D_m$ for $m>M$. To that end we locate the zeros of the auxillary function $\phi: \mig \to \mig$ with
\begin{align*}
 \phi(z) = e^z + e^{\omega_p^{p-1} z}.
\end{align*}
 For $\phi(z)=0$ to hold, we must have $|e^z |=| e^{\omega_p^{p-1} z}|$ and $\arg e^z = - \arg e^{\omega_p^{p-1} z}$ (or, equivalently, $u_0(z) = u_{p-1}(z)$ and $v_0(z) = -v_{p-1}(z)$). So for $z=x+iy$, and by \eqref{eq:gk} and \eqref{eq:hk}, we must have
\begin{align*}
e^x = \exp(x \cos(2\pi/p) + y \sin(2\pi/p)),
\end{align*}
which holds if and only if $y = \tan(\pi/p) x$, and we are interested in the ray with $x>0$ (that is to say, $V_0$) which is contained in $Q_0$. But, from \eqref{eq:args2}, for $z =x+iy \in V_0$ we have
\begin{align*}
v_{p-1} (z) = - v_0(z) = -y,
\end{align*}
and we thus get $\phi(z)=0$ for $y=m \pi + \pi/2$ with $m \in \fis$, since (as can easily be checked) these are exactly the points on $V_0$ where the arguments of $e^z$ and $e^{\omega_p^{p-1} z}$ cancel each other out. Hence the only zeros of $\phi$ in $D_m$ are $(\cot(\pi/p)(m\pi + \pi/2), m\pi + \pi/2)$ for each $m \in \fis$.

We now apply Rouché's theorem in $D_m$. We know that $\phi$ has exactly one zero inside each  $D_m$. To show that the same holds for $f$, we shall prove the inequality
\begin{equation}
|f(z) - \phi(z)| < |\phi(z)| \label{eq:rou}
\end{equation}
on each $\partial D_m$. 

Fix $m \in \fis$. We claim that $\phi$ and $f- \phi$ are symmetric about $V_0$ as well. Then, due to symmetry, it suffices to prove inequality \eqref{eq:rou} for the part of the boundary of the rectangle that lies on and to the right-hand side of $V_0$. 

For $\phi$, we need to show $\phi ( \bar{z} e^{i \pi /p}) = \overline{ \phi ( ze^{i \pi /p})}$. We have
\begin{align*}
\phi ( \bar{z} e^{i \pi /p}) = \exp ( \bar z e^{i \pi/p}) + \exp ( \bar z e^{-i2 \pi/p} e^{i \pi/p} ) =  \exp ( \bar z e^{i \pi/p}) + \exp ( \bar z e^{-i \pi/p})
\end{align*}
and
\begin{align*}
 \overline{ \phi ( ze^{i \pi /p})} = \exp ( \bar z e^{-i \pi/p}) + \overline{ \exp(z e^{-i 2 \pi/p} e^{i \pi/p} )} =  \exp ( \bar z e^{-i \pi/p}) + \exp ( \bar z e^{i \pi/p}),
\end{align*}
which proves the symmetry for $\phi$. 

For $f$, we use the function $g$ that we previously defined in Theorem \ref{3.2}, with
\begin{align*}
g(w) = f(w^{1/p}) = p \left( 1 + \frac{w}{p!} + \frac{w^2}{(2p)!} + \ldots \right).
\end{align*}
The function $g$ is entire and satisfies $g(\bar{w}) = \overline{g(w)}$. But a point $z$ and its reflection about $V_0$, say $z'$, map to a complex conjugate pair $w$ and $\overline{w}$ under $z \mapsto z^p$, thus proving the symmetry about $V_0$ for $f$. Therefore both $\phi$ and $f - \phi$ are symmetric about $V_0$ as we claimed, and we can proceed with the rest of the proof.

For $z=x+iy \in \partial D_m \cap C_m$ we have, by \eqref{eq:cot} and \eqref{eq:newarg},
\begin{equation}
|\phi(z)| = e^x + \exp(-x+2 \cot(\pi/p) m \pi), \label{eq:stavros0}
\end{equation}
since, for $z=x+iy  \in \partial D_m \cap C_m$,
\begin{align*}
  \left| e^{\omega_p^{p-1} z} \right| &= \left| \exp \left( \left( \cos(2 \pi/p) -i \sin(2 \pi/p) \right)(x+iy)  \right) \right|
\\&= \exp \left( x \cos (2 \pi/p) + y \sin (2 \pi/p) \right)
\\&=\exp \left(  x \cos(2 \pi/p) +  \left( - \cot(\pi/p)x + \frac{m \pi}{\sin^2 (\pi/p)} \right) \sin (2\pi /p) \right) \\&= \exp \left(  x \left( \cos (2 \pi/p) - \cot (\pi/p) \sin (2 \pi/p) \right) + 2 m \pi \cot (\pi/p) \right) \\&=\exp \left(  -x+2 \cot(\pi/p) m \pi \right).
\end{align*}
On this same part of $\partial D_m$, it is simple to see geometrically that
\begin{equation}
\max_{k=1, \ldots, p-2} \left\{ \left| \exp(\omega_p^k z) \right| \right\} = e^{u_1(z)} \left(=e^{u_{p-2}(z)} \right), \label{eq:stavros2}
\end{equation}
and by substituting \eqref{eq:cot} into \eqref{eq:gk} with $k=1$ we obtain
\begin{align*}
u_1(z) &= x \cos \left( \frac{ 2 \pi}{p} \right) - \left(-\cot(\pi/p) x + \frac{m \pi}{\sin^2(\pi/p)} \right) \sin  \left( \frac{ 2 \pi}{p} \right) 
\\&= x \left( \cos(2\pi/p) + \cot(\pi/p) \sin( 2 \pi /p) \right) - 2 \cot (\pi/p) m \pi
\\&= x \left( \cos(2\pi/p) + 2 \cos^2 ( \pi/p) \right)- 2 \cot (\pi/p) m \pi,
\end{align*}
so,
\begin{equation}
u_1(z) = x \left( 2\cos(2\pi/p) + 1  \right) - 2 \cot (\pi/p) m \pi. \label{eq:stavros3}
\end{equation}
Now, for $z=x+iy$ in this same part of $\partial D_m$, we can write $x = m \pi \cot(\pi/p) + c$, with $c \in (0 , \tau \sin (\pi/p))$ (as this is the perpendicular, and $\tau \sin(\pi/p)$ denotes the horizontal distance from a point in $V_0$ to $\partial D_m$). Hence, to verify \eqref{eq:rou} for $z \in D_m \cap C_m$, by \eqref{eq:stavros0}, \eqref{eq:stavros2} and \eqref{eq:stavros3}, it suffices to show that
\begin{equation}
\exp ( m \pi \cot(\pi/p) + c) + \exp ( m \pi \cot(\pi/p) - c) \label{eq:mega}
\end{equation}
is greater than $(p-2) e^{u_1(z)}$ for all $c \in (0 , \tau \sin (\pi/p))$, which we can write as 
\begin{align*}
(p-2) \exp \left( (m \pi \cot(\pi/p) + c) \left( 2\cos(2\pi/p) + 1  \right) - 2 \cot (\pi/p) m \pi \right) ,
\end{align*}
by \eqref{eq:stavros3}, that is,
\begin{equation}
 (p-2) \exp \left( m \pi \cot(\pi/p) (2 \cos(2\pi/p)-1) +c (2\cos(2\pi/p)+1) \right).\label{eq:mega2}
\end{equation}
By Lemma \ref{rexp}, for
\begin{gather*} d=(p-2)\exp(c (2\cos(2\pi/p)+1) )/(e^c+e^{-c}),
\\x=m \pi \cot(\pi/p) \text{, and}\\
a=2 \cos(\pi/p) -1,
\end{gather*}
we deduce that \eqref{eq:mega} is greater than \eqref{eq:mega2} for large enough $m \in \fis$, and we let $M$ be the largest integer for which it is not.

 For the side of $\partial D_m$ which lies in $C_{m+1}$ we can repeat the above calculations for $m+1$ instead of $m$. 

It remains to show the inequality \eqref{eq:rou} for the side of $\partial D_m$ which lies in $\partial T_0(\nu)$.

 From \eqref{eq:egine} we can deduce that 
\begin{align*}
|e^z| \geq \sum_{k=1}^{p-1} \left| e^{\omega_p^k z } \right|
\end{align*}
for $z \in \partial T_0(\nu)$. Hence
\begin{align*}
| \phi(z)| &= \left| e^z +e^{\omega_p^{p-1} z } \right| \\& \geq \left| e^z \right| - \left| e^{\omega_p^{p-1} z } \right| 
\\&\geq \sum_{k=1}^{p-1} \left| e^{\omega_p^k z } \right| - \left| e^{\omega_p^{p-1} z } \right| 
\\&= \sum_{k=1}^{p-2} \left| e^{\omega_p^k z } \right|
\\&\geq  \left| \sum_{k=1}^{p-2} e^{\omega_p^k z } \right| 
\\&= |f(z) - \phi(z)|.
\end{align*}
Thus, by Rouché's theorem, we have proven that for $m>M$, the number of zeros of $f$ in the corresponding rectangle $D_m$ is one, and we know this zero must lie on $V_0$ by Theorem \ref{3.2}. We now show that for sufficiently large values of $m$, the unique zero of $f$ in $D_m$ lies between $y=m \pi$ and $y= (m+1) \pi$. 

For the two dominant terms $e^z$ and $e^{\omega_p^{p-1} z}$, and for points $z=x+iy \in V_0$, from \eqref{eq:mods} and \eqref{eq:args2}  we have
\begin{align*}
u_0(z) = u_{{p-1}}(z) =x
\end{align*}
and
\begin{align*}
v_0(z) = -v_{{p-1}}(z) =y.
\end{align*}
Hence, for points $z =x+iy \in V_0$ with $y= 2 m \pi$, $m \in \fis$, we get (by Lemma \ref{rexp})
\begin{align*}
f(z) = 2e^x + 2\sum_{k =1,2 , \ldots, p/2-1} \exp (u_k(z)) \cos (v_k(z)) 
\end{align*}
and, for $y= 2 m \pi+\pi$, $m \in \fis$,
\begin{align*}
f(z) = -2e^x + 2\sum_{k =1, 2 , \ldots, p/2-1}\exp (u_k(z)) \cos (v_k(z)) .
\end{align*}
From Lemma \ref{rexp} we deduce that for all $p \geq 3$, there exists $x_0>0$ such that for $x> x_0$ and $z=x+iy \in V_0$ (so $y=x \tan(\pi/p)$),
\begin{align*}
e^x &> (p-2) \max_{k=1,2,\ldots,p-2} \exp(u_k(z))
 \\&=  (p-2)  \max_{k=1,2,\ldots,p-2} \exp \left( \frac{x}{\cos(\pi/p)}  \cos \left( \frac{ \pi(2k+1)}{p}\right)  \right) 
 \\&=  (p-2)   \exp \left( \frac{x}{\cos(\pi/p)}  \cos \left( \frac{ 3\pi}{p}\right)  \right),
\end{align*}
since this inequality is equivalent to $E_{p-2,a}(x)>0$ for $a=\cos(3\pi/p)/\cos(\pi/p)$. Thus, 
\begin{align*}
f(z) > 0 \text{ for } z = x+iy \in V_0 \text{ with } y=m \pi,\; m \in \fis,
\end{align*}
and
\begin{align*}
f(z) < 0 \text{ for } z=x+iy \in V_0 \text{ with } y=(m+1) \pi, \; m \in \fis.
\end{align*}
It follows from the intermediate value theorem that, for all sufficiently large $m \in \fis$, $f$ vanishes at some point in each of the segments between the points $y=m\pi$ and $y= m \pi + \pi$ on $V_0$, and thus this is the only zero of $f$ inside the corresponding rectangle $D_m$. 

The more general result, as stated, follows from obvious symmetry arguments and from the fact that the distance of $D_M$ from the origin is $M / \sin(\pi/p)$.
\end{proof}

In the following, we will use Theorem \ref{zeros} to locate the critical points of~$f$.
\begin{theorem} \label{cpsss}
All critical points of $f$ lie in $\cup_{k=0,\ldots,p-1} V_k$ and are separated in each $V_k$ from each other by the zeros of $f$. Furthermore, since those rays are mapped under $f$ onto the real axis, all the critical values of $f$ lie on the real axis, alternating between the positive and negative axes.
\end{theorem}
\begin{proof}
Let $h$ be the principal branch of $z \mapsto z^{1/p}$. Recall that, as in Theorem \ref{3.2},
\begin{align*}
g(w) = f(w^{1/p}) = p \left( 1 + \frac{w}{p!} + \frac{w^2}{(2p)!} + \ldots \right),
\end{align*}
so $g = f \circ h$ and the function $g$ is entire. We investigate its order:
\begin{align*}
\rho (f \circ h)&= \limsup_{r \to \infty} \frac{\log ( \log M(r,f \circ g))}{\log r}
\\&= \limsup_{r \to \infty} \frac{\log ( \log M(r^{1/p},f ))}{\log r}
\end{align*}
and, for $s = r^{1/p}$,
\begin{align*}
\rho (f \circ h)&= \limsup_{s \to \infty} \frac{\log ( \log M(s,f ))}{p \log s}
\\&= \frac{1}{p} \rho(f) = \frac{1}{p},
\end{align*}
since
\begin{align*}
M(r, f \circ h) = \max\{ |f(z^{1/p})|: |z|=r\}
= \max \{  |f(z)|: |z| = r^{1/p} \} = M(r^{1/p},f).
\end{align*}
We state  a theorem by Laguerre  \cite[p. 266]{tit1}.
\begin{theorem}
If $f$ is an entire function, real for real $z$, of order less than $2$, with only real zeros, then the zeros of $f'$ are also all real, and are separated from each other by the zeros of $f$.
\end{theorem}
Since $g$ is real on $\pra$, we can apply Laguerre's theorem to $g=f \circ h$, which certainly has order less than $2$, and whose zeros all lie on the negative real axis, from the proof of Theorem \ref{3.2}. We deduce that the critical points of $g=f \circ h$ are all real and are separated by the zeros of $g$. Hence, by symmetry, the critical points of $f$ lie on $\cup_{k=0}^{p-1} V_k$ and on each $V_k$ they are separated by the zeros of $f$. \end{proof}

\section{Trapeziums}\label{sec:trapeziums}
In this section we locate compact subsets of the plane that cover themselves under $f$. Inside these compact sets we construct invariant Cantor sets for $f$ on which $f$ is conjugate to the one-sided shift on $\Sigma_N$, and which will serve as the endpoints of the curves in the Cantor bouquet. Specifically, these compact sets are the trapeziums $T_{m,c}$ bounded by the lines
\begin{align*}
y&= \tan (\pi/p) x \\
y&=(2m-1) \pi \\
y&=(2m+1) \pi \\
x&=c,
\end{align*}
for $m \in \fis$, for large enough $c> \tau/ \sin(\pi/p)$, and with their sides labelled, respectively, as $S_{m,c}^1, S_{m,c}^2, S_{m,c}^3, S_{m,c}^4$. Note that we define the sets $T_{m,c}$ to contain the boundary as well as the inside of each trapezium.  We also define $T_{-m,c}$ to be the reflection of $T_{m,c}$ with respect to $\pra$.

We know from Theorem \ref{zeros} that, for large enough $m$, $f$ maps $S_{m,c}^1$ into a bounded subset of the real axis which contains $0$ (which $f$ attains twice on $S_{m,c}^1$).

We proceed to investigate the curves $f(S_{m,c}^2)$ and $f(S_{m,c}^3)$. We give a lemma about the behaviour of $f$ at the points on the half-lines $Y^{\pm}_m:=\{ x+iy \in \mig : y = (2m\pm 1)\pi, x > \cot(\pi/p)y\}$ for $m \in \fis$, which are the horizontal half-lines that contain a segment of the boundary of the trapezium $T_{m,c}$ (analogous results immediately follow for $T_{-m,c}$).

\begin{lemma} \label{Ym}
There exists $M \in \fis$ such that $f(Y^{\pm}_m)$ lies in the left half-plane for all $m > M$.
\end{lemma}
\begin{proof}
Let $z= x+iy$. We prove the result for $Y^{+}_m$; $Y^{-}_m$ is similar. For $x+iy \in Y^{+}_m$, we will write $x= \cot(\pi/p) y_m + \xi$ where $y_m = (2 m +1) \pi$ and $\xi>0$. Assuming that $x+iy \in Y_m$ in the following, we have, from (\ref{eq:gk}) and (\ref{eq:hk}), for $k=0,1 , \ldots, p-1$, 
\begin{equation} \label{eq:4.1}
u_k(z) = \frac{y_m}{\sin(\pi/p)} \cos \left( \frac{\pi(2k+1)}{p} \right) + \xi \cos  \left( \frac{2k\pi}{p} \right)
\end{equation}
and
\begin{equation}\label{eq:4.2}
v_k(z) = \frac{y_m}{\sin(\pi/p)} \sin \left( \frac{\pi(2k+1)}{p} \right) + \xi \sin  \left( \frac{2k\pi}{p} \right),
\end{equation}
which we now treat as functions of $\xi$. For $k \in \{0 , \ldots, p-1\}$ we define the functions $\phi_k : (0, +\infty)  \to \pra$ by
\begin{align*}
\phi_k( \xi ) = \exp (u_k(z)) \cos (v_k(z)),
\end{align*}
so,
\begin{align*}
\operatorname{Re} f (z) = \sum_{k=0}^{p-1} \exp(u_k(z)) \cos (v_k(z))  =\sum_{k=0}^{p-1} \phi_k(\xi) .
\end{align*}
We want to prove that $\operatorname{Re} f (z)  < 0$ for all points in $Y_m$ when $m$ is sufficiently large. The two largest terms with respect to their moduli for small $\xi$ are the terms corresponding to $k=0$ and to $k = p-1$. By \eqref{eq:4.1} and \eqref{eq:4.2} these are, respectively,
\begin{equation}
\phi_0 (\xi) = - \exp(  y_m \cot( \pi/p) + \xi), \label{eq:4.3}
\end{equation}
which is negative for all $\xi>0$, and (since $\sin( \pi(2(p-1)+1)/p) = - \sin(\pi/p)$ and $\sin(2(p-1)\pi/p) = - \sin(2 \pi/p)$)
\begin{align*}
\phi_{p-1}(\xi)=\exp(y_m\cot(\pi/p) + \xi \cos(2\pi/p) ) \cos ( y_m  + \xi \sin (2\pi/p)),
\end{align*}
which is negative if $\xi$ satisfies
\begin{align*}
\pi<\pi  + \xi \sin (2\pi/p) < 3 \pi/2;
\end{align*}
that is, if
\begin{align*}
0 <\xi  < \frac{\pi}{2 \sin( 2 \pi/p)}.
\end{align*}
But for $\xi  \geq \pi/(2 \sin( 2 \pi/p))$ we have
\begin{align*}
\left| \frac{ \phi_{p-1}(\xi)}{ \phi_{0}(\xi)} \right| &\leq \frac{ \exp(y_m \cot(\pi/p) + \xi \cos(2 \pi/p))}{\exp(y_m \cot(\pi/p) + \xi)} 
\\&=  \exp (\xi ( \cos(2\pi/p) -1)),
\end{align*}
which is a function of $\xi$ that decreases to $0$ as $\xi$ increases to $\infty$, and thus attains its maximum value for $\xi  = \pi/(2 \sin( 2 \pi/p))$. There consequently exists $\mu=\mu(p) \in (0,1)$ such that
\begin{align*}
|\phi_{p-1}(\xi)| \leq \mu  |\phi_0(\xi)|, \text{ for } \xi \geq\pi/(2 \sin( 2 \pi/p)),
\end{align*}
and thus, since $\phi_0(\xi)$ is negative for all $\xi>0$,
\begin{equation}
\phi_0(\xi) + \phi_{p-1}(\xi) \leq (1-\mu) \phi_0(\xi) \label{eq:mi1} \text{, for all } \xi>0.
\end{equation}
It is simple to see that, for all $\xi>0$, 
\begin{equation}
\max_{k=1, \ldots, p-2} |\phi_k(\xi)| = |\phi_1(\xi)| = \exp \left( \frac{y_m \cos (3\pi/p)}{\sin(\pi/p)} + \xi \cos (2 \pi /p) \right), \label{eq:4.5}
\end{equation}
so
\begin{align*}
 \sum_{k=1}^{p-2} |\phi_k(\xi)|  \leq (p-2) |\phi_1(\xi)|, \text{ for } \xi >0.
\end{align*}

We now use Lemma \ref{rexp}. Following its notation, we substitute 
\begin{align*}x = y_m \cot(\pi/p),
\end{align*}
as well as
\begin{align*}d=2(p-2) \exp(\xi (\cos(2\pi/p)-1))/(1-\mu)\end{align*}
 and 
\begin{align*}a = \cos(3\pi/p)/ \cos(\pi/p)<1,
\end{align*}
 and choose $M$ so large that $e^x-de^{ax}>0$ for all $m > M$. Returning to the original notation of this part, it follows that for all $m>M$ and $x+iy \in Y_m$, we have
\begin{equation*}
 (p-2) |\phi_1(\xi)| \leq \frac{1-\mu}{2} |\phi_0(\xi)|,
\end{equation*}
by \eqref{eq:4.3} and  \eqref{eq:4.5}, and thus
\begin{equation}
 \sum_{k=1}^{p-2} |\phi_k(\xi)|  \leq  \frac{1-\mu}{2} |\phi_0(\xi)| \text{ for } \xi>0 \label{eq:mi2},
\end{equation}
by  \eqref{eq:4.5} again. Hence, finally, for $z \in Y_m$ with $m>M$, from (\ref{eq:mi1}) and (\ref{eq:mi2}) we have
\begin{align*}
\operatorname{Re} f (z) = \sum_{k=0}^{p-1} \phi_k(\xi)  \leq \frac{1-\mu}{2} \phi_0(\xi) <0 \;\text{ for } \xi >0,
\end{align*} 
as required.\end{proof}

From Lemma \ref{Ym}, we know that $f(S_{m,c}^2)$ and $f(S_{m,c}^3)$ lie in the left half-plane for $m>M$. We now investigate the behaviour of $f(S_{m,c}^4)$ in order to complete the proof that the trapeziums cover themselves under $f$. Specifically, we prove the following.

\definecolor{wrwrwr}{rgb}{0.3803921568627451,0.3803921568627451,0.3803921568627451}
\begin{lemma} \label{trap}
 For large enough $c>0$, the set $\{ z \in \mig: \operatorname{Re}z \geq 0 \} \cap f(S_{m,c}^4)$ is a curve that meets both the positive and negative imaginary axes and lies inside an annulus the form 
$ A_{m,c}:=\{ z: \left| |z| - e^c \right| \leq  \max_{z \in S_{m,c}^4} |(f-\exp) (z)| \}$.
\end{lemma}
\begin{proof}The image of $S_{m,c}^4$ under $e^z$ is the circle around the origin with radius $e^c$. Additionally, by Lemma \ref{symmetry} and  \eqref{eq:egine},
\begin{align*}
\max_{z \in S_{m,c}^4} |(f-\exp)(z)| &=\max_{z \in S_{m,c}^4} \sum_{k=1}^{p-1} \exp ( u_k(z))
\\&\leq (p-1) \max_{z \in S_{m,c}^4} \exp( u_{p-1}(z)) 
\\&=  (p-1) \max_{z \in S_{m,c}^4}  \exp \left(  c \cos \left( \frac{ 2 \pi}{p} \right) + y \sin  \left( \frac{ 2 \pi}{p} \right)  \right)
\\&=  (p-1) \exp\left( (2m+1) \pi  \sin  \left( \frac{ 2 \pi}{p} \right)  \right)  \exp \left(  c \cos \left( \frac{ 2 \pi}{p} \right) \right).
\end{align*}
But, since $\cos(2 \pi/p)<1$, the quantity $\max_{z \in S_{m,c}^4} |(f-\exp)(z)|$ is small compared to $e^c$, since
\begin{align*}
\frac{\exp (c)}{\exp(c \cos (2\pi/p))} = \exp( c(1-\cos(2\pi/p)) \to \infty, \text{ as } c \to \infty.
\end{align*}
Thus, for large enough $c$, $f(z)$ for $z \in  S_{m,c}^4$ has to lie inside a ball of radius $\max_{z \in S_{m,c}^4}  |(f-\exp)(z)|$ around a point on $\{z:|z|=e^c\}$. 
For large enough $c$, then, $f(S_{m,c}^4)$ is a curve inside an annulus that meets both the positive and negative imaginary axes, and joins up the rest of the image of $\partial T_{m,c}$ under $f$ in the left half-plane.
\end{proof}

We can now define inverse branches of $f$ inside the trapeziums $T_{m,c}$, using the following lemma.

\begin{lemma}
There exists $c>0$ such that the inverse branch of $f$ in $T_{m,c}$ is well defined and analytic for large enough $m$.
\end{lemma}
\begin{proof}
Let $r(m,c)$ denote the radius of the inner boundary curve of the annulus $A(m,c)$ that was defined in Lemma \ref{trap}; that is,
\begin{align*}
r(m,c) : =e^c-\max_{z \in S_{m,c}^4}  |(f-\exp)(z)|.
\end{align*}
 Fix $K \in \fis$. Choose $c_0>1$ such that $T_{i,c}$ and $T_{j,c}$ exist for $c \geq c_0$, with $T_{i,c} \subset f(T_{j,c})$ for all $i,j \in \{ M, \ldots, M+K-1, -M, \ldots, -M-K+1 \}$; Lemmas~\ref{Ym} and \ref{trap} guarantee that such a $c_0>1$ exists. The image curve $f(S_{m,c}^4)$ goes round $A(m,c)$, so it surrounds $\{z:\operatorname{Re}z>0\} \cap B(0,r(m,c))$. We can thus deduce from Rouché's theorem that the points inside $\{ z : \operatorname{Re}z > 0\} \cap B(0, r(c,m))$ are covered by the image of $T_{m,c}$ under $f$ as many times as they would be covered under $\exp$; that is, once. Hence, the inverse of $f$ in $T_{m,c}$ is well defined.
\end{proof}

The values $M$ and $K$ will remain fixed from now on, and for the following we assume that $c=c_0$ and write $T_{M+j,c}=T_{j+1}$, $T_{-M-j,c}=T_{-j-1}$ for $j \in \{ 0, \ldots, K-1\}$ and $S_{m,c}^i=S_m^i$. Let $L_j$ be the branch of the inverse of $f$ on $\{ z \in \mig: \operatorname{Re}z > 0\} \cap B(0, r(c,m))$ that takes values in $T_j$.

 Let $T^K = \cup_{ 1 \leq |j| \leq K} T_{j}$ and let $\Lambda_{K}$ be the set of points whose orbits remain for all time in $T^K$ (see Figure 1). Recall that the trapeziums are closed sets.

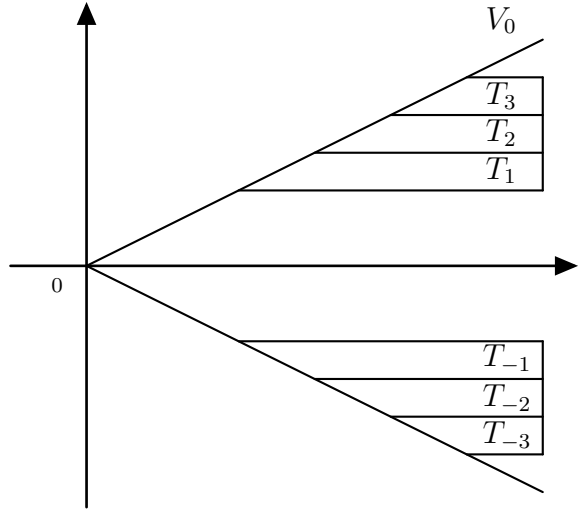
\begin{figure}[h]
\centering
\begin{subfigure}[b]{.8\textwidth}
\begin{tikzpicture}[line cap=round,line join=round,>=triangle 45,x=1cm,y=1cm]
\clip(-2.8045585111180515,-5.193966821972128) rectangle (13.541482789665318,5.290615224489944);
\draw [line width=0.8pt] (6,2)-- (4,2);
\draw [line width=0.8pt] (6,2.5)-- (5,2.5);
\draw [line width=0.8pt] (6,1.5)-- (3,1.5);
\draw [line width=0.8pt] (6,2.5)-- (6,1);
\draw [line width=0.8pt] (6,-2)-- (4,-2);
\draw [line width=0.8pt] (6,-1.5)-- (3,-1.5);
\draw [line width=0.8pt] (5,-2.5)-- (6,-2.5);
\draw [line width=0.8pt] (6,-2.5)-- (6,-1);
\draw (5.107035552813294,1.5893651319724926) node[anchor=north west] {$T_1$};
\draw (5.107035552813294,2.0846997168447166) node[anchor=north west] {$T_2$};
\draw (5.107035552813294,2.5800343017169403) node[anchor=north west] {$T_3$};
\draw (5.107035552813294,3.5800343017169403) node[anchor=north west] {$V_0$};
\draw (5.0657576707406085,-0.89835484983643985) node[anchor=north west] {$T_{-1}$};
\draw (5.0657576707406085,-1.4239202593335363) node[anchor=north west] {$T_{-2}$};
\draw (5.0657576707406085,-1.9192548442057602) node[anchor=north west] {$T_{-3}$};
\draw [line width=0.8pt] (0,0)-- (6,-3);
\draw [line width=0.8pt] (0,0)-- (6,3);
\draw [line width=0.8pt] (6,-1)-- (2,-1);
\draw [line width=0.8pt] (6,1)-- (2,1);
\draw [->,line width=1pt] (0,0) -- (6.464936032454278,0);
\draw [line width=1pt] (0,0)-- (-1,0);
\draw [->,line width=1pt] (0,0) -- (0,3.5);
\draw [line width=1pt] (0,0)-- (0,-3.2);
\begin{scriptsize}
\draw(-0.38759770250294356,-0.26159325863444793) node {$0$};
\end{scriptsize}
\end{tikzpicture}
\end{subfigure}\caption{The set $T^K$ for $K=3$.}
\end{figure}

\begin{theorem} \label{lj}
The set $\Lambda_K$ is homeomorphic to $\Sigma_K$ and $f_{|\Lambda_K}$ is conjugate to the shift map on $\Sigma_K$.
\end{theorem}
\begin{proof}
Let $s = s_0 s_1 s_2 \ldots \in \Sigma_K$ and define
\begin{equation}
L_s^n (z)= L_{s_0} \circ \cdots \circ L_{s_{n-1}}(z), \quad \text{ for } z \in T^K. \label{eq:asteraki}
\end{equation}
We claim that, for $z \in T^K$,
\begin{align*}
\lim_{n \to \infty} L_s^n (z)
\end{align*}
exists and is independent of $z$. 

For any $1\leq |j| \leq K$ consider the image set $f(\operatorname{int} T_j )$;  from Lemmas \ref{Ym} and \ref{trap} and the remarks before them it follows that this image set will lie inside $B(0, r(c,m))$ and cover $T^K$. This allows us to choose a simply connected region $G$ inside $f(\operatorname{int} T_j )$, for any $1\leq |j| \leq K$, so that $T^K  \subset G$. Thus, for each $j$, there exists an open connected subset of $T_j$, say $T^{'}_j$, which maps univalently onto $G$ under $f$. Then the inverse branch $L_{s_j}$ maps $T^{'}_j$ strictly inside itself and so each $L_{s_j}$ is a strict contraction of the Poincaré metric on $G$. In particular, each $L_{s_j}$ is a uniformly strict contraction of the Poincaré metric on $T^{K{'}}=  \cup_{ 1 \leq |j| \leq K} T^{'}_{j}$ (since there are only finitely many $L_{s_j}$). Therefore, the sets $L^n_s(T^{K{'}})$ are nested and their diameters decrease to $0$ as $n \to \infty$. So $\lim_{n \to \infty} L_s^n (z)$ exists and is independent of $z$.

We can thus define $\Phi(s) = \lim_{n \to \infty} L_s^n (z)$ for all $z \in T^{K{'}}$. A standard argument (see, for example, \cite[Theorem 9.9]{blanchard}) then shows that $\Phi$ is a homeomorphism, which gives the conjugacy between $f$ and the shift map.
\end{proof}

For each $s \in \Sigma_k$, let $z(s)$ be the unique point in $\Lambda_K$ whose itinerary is $s$.

\begin{cor} \label{ela} Let $s = \overline{s_0 s_1 \cdots s_{n-1}}$ be a repeating sequence in $\Sigma_K$. Then $z(s)$ is a repelling periodic point of $f$ with period $n$.
\end{cor}
\begin{proof} 
The map $L_s^n$ is a composition of analytic maps and therefore analytic itself. Also, $L^n_s (T_{s_0})$ is contained in the interior of $T_{s_0}$. Since $L^n_s$ is a strict contraction of the Poincaré metric on $T_{s_0}$, it follows that $T_{s_0}$ has a unique fixed point in this trapezium and that this fixed point is attracting for $L^n_s$; thus repelling for $f$. Since this point has itinerary $s$ for $f$, it must be $z(s)$.
\end{proof}

\begin{cor} Let $s \in \Sigma_K$. Then $z(s) \in J(f)$. \label{jcor}
\end{cor}
\begin{proof} 
From Corollary \ref{ela} it follows that $z(s)$ is a limit of repelling periodic points given by the conjugacy with the shift map. By a result of Baker \cite{baker2}, $J(f)$ is the closure of the set of repelling periodic points of $f$; hence $z(s) \in J(f)$.
\end{proof}

We note that $z(s)$ is not the only point inside the strip $R(s_0)$ that has itinerary $s$: in fact, there are infinitely many points in this strip that share the itinerary $s$ and form a curve, as we will show in the next section.

To end this section, we prove an additional property of the points $z(s)$, $s \in \Sigma_K$; some of them will always exist in $R(m) \cap Q_0$ and $R(-m) \cap Q_1$ for large enough positive $m$.

\begin{lemma} \label{curveswd}
Let $f \in \mathcal{F}$. For $m \geq M$, there exist $s, s' \in \Sigma_K$, such that $z(s) \in R(m) \cap Q_0$ and $z(s') \in R(-m) \cap Q_1$.
\end{lemma}
\begin{proof}
Consider the repeating sequence $s=\overline{s_j}$ in $\Sigma_K$. Then $z(s) \in R(j)$. If $z(s)$ was in $T_0(\nu)$, from Lemma \ref{davel} we would have $f^n(z(s)) \to \infty$ as $n \to \infty$, which is a contradiction. Therefore, $z(s) \in Q_0 \cup Q_1$.
\end{proof}

\section{Hairs}\label{sec:hairs}
We continue to use the proof strategy of \cite{dev3} in this section, with our goal being to show that each point in $\Lambda_K$ (constructed in Section 4) actually lies at the endpoint of a continuous curve, all points of which share the same itinerary.

\begin{definition} \label{hairdef}
Let $s=s_0 s_1 s_2 \ldots \in \Sigma_K$. A continuous curve $h_s: [1 , \infty) \to R_{s_0}$ is called a \emph{hair} of $f$ that is attached to $z(s)$ if
\begin{enumerate}
\item $h_s ( 1) = z(s)$;
\item for each $t \geq 1$, the itinerary of $h_s(t)$ under $f$ is $s$;
\item if $t>1$, then $\lim_{n \to \infty} \operatorname{Re} f^n (h_s(t)) = \infty$;
 \item $\lim_{t \to \infty} \operatorname{Re} h_s(t) = \infty$;
\end{enumerate}
or, if it is a rotation of a curve that satisfies the above properties.
\end{definition}
In the following we continue to focus our attention on $T_{\nu} (0)$, since analogous results follow for the other angles due to symmetry. A hair attached to $z(s)$ is a continuous curve that extends from the endpoint $z (s)$ to infinity in the right half-plane. Any point $z$ on this hair that is not $z(s)$, shares the same itinerary as $z(s)$, and has an orbit that tends to infinity in the right half-plane. Further, the orbit of the endpoint $z(s)$ is bounded since $s \in \Sigma_K$.

To show that such hairs exist, we need a covering result which we can obtain as a corollary to two lemmas from the previous section.

\begin{cor} Let $H_{m,c}$ be the half-strip that extends to infinity in the right half-plane, bounded by the lines
\begin{align*}
y&=(2m-1) \pi \\
y&=(2m+1) \pi \\
x&=c,
\end{align*}
for $m \in \fis$ and $c>0$. Then, for large enough $c>0$, we have
\begin{align*}
\{z \in \mig: \operatorname{Re} z> 0 \}  \cap f(H_{m,c}) \subset \{z \in \mig: \operatorname{Re} z> 0 \}  \cap B (0, R(c))^{\mathsf{c}},
\end{align*}
where $R(c) \sim e^c$ as $c \to \infty$. \label{corinv}
\end{cor}
\begin{proof} This is an immediate consequence of Lemmas \ref{Ym} and \ref{trap}, keeping in mind that we can apply Lemma \ref{trap} to all large enough $c$.
\end{proof}
Following the notation established in the previous section, we extend the inverse branches of $f$, $L_j$ (previously defined in the proof of Theorem \ref{lj}), to the half-strips $H_{m,c}$ as follows: 

Let $z \in H_{m,c}$ for some $m$ and $c$. Then $L_j(z)$ is the preimage under $f$ of $z$ in $H_{j,c}$. Note that, for all $m$ and $c$ and an~ $j$, we have
\begin{equation}
\left| L_j^{'} (z) \right| < 1/2 \text{ for } z \in H_{m,c}, \label{eq:daveanisotitav2} 
\end{equation}
by \eqref{eq:daveleq1}.

We will now prove that if $s=s_0 s_1 s_2 \ldots $ is a bounded sequence, then there is a unique hair attached to $z(s)$.
Let $E(z) = (1/e) e^z$. 
For any $s \in \Sigma_K$ we define the functions $G^n_s : [1 , \infty) \to \mig$ by
\begin{align*}
G^n_s (t) = L^n_s \circ E^n(t), \quad 1 \leq t < \infty;
\end{align*}
recall that $L_s^n$ was defined in \eqref{eq:asteraki}.
We will show that the limit function of $G^n_s$ exists and that it provides a parametrisation of the hair $h_s$ as a function of $t$. First, an inequality about the functions $G^n_s (t)$.

\begin{prop} \label{anisotita}
There exist $q, M>0$ such that, for any $s=s_0 s_1 s_2 \cdots \in \Sigma_K$ and for all $t > q+1$ and $n \geq 1$, we have
\begin{equation}
t- M \leq \operatorname{Re} G^n_s (t) \leq t + M.
\end{equation}
\end{prop}
 \begin{proof}
Since $|s_i| \leq K$ for all $i$, there exists $M_K > 2 \pi$ such that $|\operatorname{Im} L_{s_i}(z)| < M_K$ for each $s_i$ and all $z=x+iy \in T_0(\nu)$, whose preimages we will consider. 

Define $\epsilon(z) := f(z)/ e^z-1$. We will make use of the following estimate of the quantity $|1+\epsilon(z)|$:
\begin{align*}
\left| 1 + \epsilon(z) \right|   &=  \left| \frac{ e^z + \sum_{k=1}^{p-1} \exp(\omega_p^k z)}{e^z} \right|
\\&=  \left|  1+ \sum_{k=1}^{p-1} \exp((\omega_p^k-1) z) \right|
\\&\leq 1+ (p-1)  \max_{{k \in \{1, \ldots, p-1\}}} \left| \exp ((\omega_p^k -1 ) z \right|
\\&\leq 1+ (p-1)  \max_{{k \in \{1, \ldots, p-1\}}} \exp \left( \operatorname{Re}((\omega_p^k -1 ) z)  \right)
\\&\leq 1+ (p-1) \exp  \left( \max_{{k \in \{1, \ldots, p-1\}}} \left( (\cos(2 k\pi/p)-1)x - y(\sin(2k\pi/p)-1) \right) \right).
\end{align*}
We can thus write
\begin{equation} \label{eq:epsa}
\left| 1 + \epsilon(z) \right|  \leq 1+ Ce^{a x},
\end{equation}
where
\begin{align*}
a=  \max_{{k \in \{1, \ldots, p-1\}}}  \left(  \cos( 2k \pi /p)- 1 \right) <0
\end{align*}
and
\begin{align*}
C = (p-1)  \max_{k \in \{1, \ldots, p-1\}, |y| \leq M_K} \exp \left( -y (\sin (2 \pi k/p) -1)\right)  >0.
\end{align*}
Recall that $E: \pra \to \pra$ is $E(t)=(1/e)e^t$.
We define 
\begin{align*}
s_n(t):= \sum_{k=0}^{n} \log ( 1 + C \exp (a E^k(t)))
\end{align*}
for $t>1$, with $s(t) : = \lim_{n \to \infty}s_n(t)$. The series defining $s(t)$ is convergent and the function $s$ is decreasing to $0$ with respect to $t$, since $a<0$ and $C=C(p,k)$ is independent of $y$. 

In the following we will consider several lower bounds for $q$, starting here: we can choose $q>0$ large enough and, further, $M=M(q)>1$ such that
\begin{equation}
s_n(t) \leq s(t) \leq M - 1 \label{eq:epsa2}
\end{equation}
and
\begin{equation}
s_n(t) \leq s(t) + \log M_K - 1 \leq M,  \label{eq:epsa4}
\end{equation}
for all $t > q + 1$.
We further choose $q> 0$ to be large enough so that
\begin{equation} \label{eq:dyo}
E(q)  >  M + 2 
\end{equation}
and consider all $t>q+1$ such that
\begin{equation}
\operatorname{Re} L_j (E(t)) \geq q \label{eq:epsa1.5}
\end{equation}
for all $|j| \leq K$. 

For $t \geq 1$, we have $E(t) \geq 1$, so
\begin{equation}
1 + \frac{s_n(q)}{E(t)} \leq e^{s_n(q)}, \quad \text{ for all } n \in \fis \text{ and } q>0.  \label{eq:epsa5}
\end{equation}

 For the rest of the proof we assume that $q$ is large enough so that (\ref{eq:epsa2}), (\ref{eq:epsa4}), (\ref{eq:epsa1.5}) and (\ref{eq:epsa5}) all hold.
We will prove that for any sequence $s \in \Sigma_{K'}$ (with $K' \leq K-1$) it is the case that
\begin{align*}
t-M \leq \operatorname{Re} G^n_s (t) 
\end{align*}
for all $n \in \fis$ and all such large enough $t>q+1$.\\
We have 
\begin{align*}
f(z) = e^z (1+\epsilon(z)),
\end{align*}
so
\begin{align*}
e^z = \frac{f(z)}{1+\epsilon(z)}
\end{align*}
and thus, if $f(z) = w$, with $z,w \in T_0(\nu)$, then the corresponding inverse branch is
\begin{equation} \label{eq:stavros}
f^{-1}(w) = z = \log w - \log(1+ \epsilon(z))
\end{equation}
for the appropriate branch of the logarithm.
Hence we can write
 \begin{align*}
\operatorname{Re} G^1_s (t)&=  \operatorname{Re} L_{s_0} (E(t))
\\&= \operatorname{Re} \log E(t) -  \operatorname{Re} \log(1+  \epsilon( L_{s_0}( E(t))))
\\&\geq t-1 - \log (1+Ce^{a q})
 \\&\geq t - 1 - s_0(q)
 \\&\geq t - M
 \end{align*}
with the first inequality following from (\ref{eq:epsa}) and (\ref{eq:epsa1.5}), the second following from the fact that $s_0(q) \leq s(q)$, while the third follows from (\ref{eq:epsa2}). This is the first step of the induction. Now let us assume that, for all $s \in \Sigma_k$, for some $m \geq 3$ and for all $q$ large enough and $t > E(q)$, we have
 \begin{equation}
\operatorname{Re} G_s^m(t) \geq t -1 - s_{m-1}(q)\label{eq:in1},
\end{equation}
from which
 \begin{align*}
\operatorname{Re} G_s^m(t) \geq t - M
\end{align*} follows, by \eqref{eq:epsa2} and the definition of $s(q)$. We will proceed to deduce that \eqref{eq:in1} holds with $m$ replaced by $m+1$.
 Substituting $E(t)$ for $t$ and $E(q)$ for $q$ in (\ref{eq:in1}), we obtain
 \begin{align*}
\operatorname{Re} G_s^m(E(t)) &\geq E(t) - 1-s_{m-1}(E(q))
\end{align*}
from which it follows by \eqref{eq:epsa2} and the definition of $s(q)$ that
 \begin{equation} \label{eq:dyo2}
\operatorname{Re} G_s^m(E(t)) \geq E(t) - M.
\end{equation}
So, by \eqref{eq:epsa} and \eqref{eq:stavros},
 \begin{align*}
\operatorname{Re} G_{s}^{m+1} (t) &= \operatorname{Re} L_{s_0} \left( G_{\sigma(s)}^m (E(t)) \right) 
\\& \geq \log \left| G_{\sigma(s)}^m (E(t)) \right| - \log(1+Ce^{a q})
\\&\geq \log \left| \operatorname{Re} G_{\sigma(s)}^m (E(t)) \right| - \log(1+Ce^{a q})
\\&= \log \left( \operatorname{Re} G_{\sigma(s)}^m (E(t)) \right) - \log(1+Ce^{a q}),
\end{align*}
with the last equality following from \eqref{eq:dyo} and \eqref{eq:dyo2}. We claim that
 \begin{align*}
\operatorname{Re} G_{s}^{m+1} (t) &\geq \log \left(  E(t) - 1 - s_{m-1}(E(q)) \right) -  \log(1+Ce^{a q})
\\&\geq   \log E(t) - \log(1+s_{m-1}(E(q)))-  \log(1+Ce^{a q})
\\&\geq \log E(t)  - s_{m-1}(E(q)) - \log(1+Ce^{aq}) 
\\&= \log E(t) - s_m(q) 
\\&\geq t  -1-s_m(q)
\\&\geq t  - M,
\end{align*}
for any $n \in \fis$ and $t>E(q)$.
The third inequality follows easily from the fact that $\log(1+x) \leq x$ for $x>0$. We now prove that the second one holds as well. 

Here we use the fact that, for $a>b>1$ and $a \geq b^2/(b-1)$, we have
 \begin{align*}
\log (a-b) \geq \log a - \log b,
\end{align*}
which we apply with $a=E(t)$ and $b=1+s_{m-1}(E(q))$.
For $t > E(q)$, we have 
 \begin{align*}
E(t)>1+s_{m-1}(E(q))>1,
\end{align*}
so it remains to show that, for large enough $q$, (and since $t >E(q)$)
 \begin{align*}
E(E(q)) \geq \frac{(1+s_{m-1}(E(q)))^2}{s_{m-1}(E(q))},
\end{align*}
or, equivalently,
 \begin{align*}
E(E(q)) \geq 1/s_{m-1}(E(q)) + 2 + s_{m-1}(E(q)).
\end{align*}
The quantity $s_{m-1}(E(q))$ decreases to $0$ as $q$ increases to infinity, but
 \begin{align*}
1/s_{m-1}(E(q)) \leq 1/s_0(E(q)) = 1/\log(1+Ce^{aE(q)}).
\end{align*}
For $q$ large enough we have $Ce^{a E(q)} \leq 1$, since $C$ is bounded and $-1<a<0$, so
 \begin{align*}
\frac{1}{\log(1+Ce^{a E(q)})} \leq \frac{Ce^{-aE(q)}}{\log 2} \leq E(E(q)),
\end{align*}
again for $q$ large enough, and using the fact that 
 \begin{align*}
\frac{\log(1+x)}{x} \geq \log 2
\end{align*}
for $0<x<1$.

We have thus proven our claim that \eqref{eq:in1} holds with $m+1$ replacing $m$ and so the induction is complete.

We will now, again using induction, prove that
 \begin{align*}
\operatorname{Re} G^n_s(t) \leq t + M
\end{align*}
for all $n \geq 1$, and for $q$ large enough and $t > q + 1$.
From (\ref{eq:epsa}) and  (\ref{eq:epsa2}) we have
 \begin{align*}
\operatorname{Re} G^1_s (t)&=  \operatorname{Re} L_{s_0} (E(t))
\\&= \operatorname{Re} \log E(t) -  \operatorname{Re} \log(1+  \epsilon( L_{s_0} (E(t))))
\\&\leq t-1 + \log(1+Ce^{aq})
\\&\leq t + M.
\end{align*}
Now suppose that
 \begin{align*}
\operatorname{Re} G^m_s(t) \leq t+s_m(q),
\end{align*}
so, in particular,
 \begin{align*}
\operatorname{Re} G^s_u(t) \leq t + M.
\end{align*}
Then,
\begin{align*}
\operatorname{Re} G^m_{\sigma(s)}(E(t)) \leq E(t) +  \sum_{k=1}^m \log(1+C \exp(a E^k(q))),
\end{align*}
and
\begin{align*}
\operatorname{Re} G_{s}^{m+1} (t) &= \operatorname{Re} L_{s_0} \left( G_{\sigma(s)}^m (E(t)) \right).
\end{align*}
From this, together with \eqref{eq:stavros}, we have
\begin{align*}
\operatorname{Re} G_s^{m+1} (t) &\leq \log | G^m_{\sigma(s)} (E(t)) | +\log(1+Ce^{a{q}}).
\end{align*}
Thus,
\begin{align*}
\operatorname{Re} G_s^{m+1} (t) &\leq \log \left( | \operatorname{Re} G^m_{\sigma(s)} (E(t)) | + | \operatorname{Im} G^n_{\sigma(s)} (E(t)) |  \right)+  \log(1+Ce^{aq})
\\&\leq \log(  \operatorname{Re}  G^m_{\sigma(s)} (E(t))) + \log(1+Ce^{aq}) + \log M_K,
\end{align*}
since $\log(a+b) \leq \log a + \log b$ as long as $a,b >2$ (recall that $\operatorname{Re}  G^m_{\sigma(s)} (E(t)) \geq E(t) - M >2$ from \eqref{eq:dyo} and \eqref{eq:dyo2}, as well as that $M_K > 2 \pi$). Now, by taking logarithms in \eqref{eq:epsa5}, we have
\begin{align*}
\operatorname{Re} G_s^{m+1} (t) &\leq \log \left(   E(t) +  \sum_{k=1}^m \log(1+C \exp(a E^k(q))) \right)+ \log(1+Ce^{aq}) + \log M_K
\\&\leq t-1 +\sum_{k=1}^m \log(1+C \exp(a E^k(q))) + \log(1 + Ce^{aq}) + \log M_K
\\&=  t-1 +\sum_{k=0}^m \log(1+C \exp(a E^k(q))) + \log M_K
\\&=  t-1 + s_m(q)+ \log M_K
\\&\leq t+M,
\end{align*}
thus proving the desired result for $m+1$ and completing the induction.
\end{proof}

We now prove that 
\begin{align*}
h_s(t) : = \lim_{n \to \infty} G^n_s (t)
\end{align*}
is a well defined function for $t \geq 1$. It suffices to prove that $\{G^n_s (t)\}$ is Cauchy for all large $t$. From the result of Proposition \ref{anisotita} and the quantity $M_K$ defined at the start of its proof, we have, for large enough $t$,
\begin{align*}
|G^n_s (t) - G^{n+1}_su (t) | \leq 2(M+M_K) 
\end{align*}
for any $s \in \Sigma_K$. For those large enough $t$ we have
\begin{align*}
\left| G^{N+n}_s (t) - G^{N+n+1}_s (t) \right| &= \left| L^N_s \circ G^{n}_{\sigma^N(s)} (t) - L^N_s \circ G^{n+1}_{\sigma^N(s)} (t)  \right| 
\\& \leq   \left| \left( L^N_s \right)^{'} (z) \right|  \left| G^{n}_{\sigma^N(s)} (t) - G^{n+1}_{\sigma^N(s)} (t)  \right| 
\\& \leq  (1/2)^N  2(M+M_K),
\end{align*}
with the last inequality due to $\eqref{eq:daveanisotitav2}$. 
Now let $\varepsilon>0$. There exists $N=N(\varepsilon)>0$ such that, for all $m>n\geq N$, 
\begin{align*}
\left| G^{N+n}_s (t) - G^{N+m}_s (t) \right| \leq 2(M+M_K) \sum_{k=0}^{m-n-1} \frac{1}{2^{N+k}},
\end{align*}
which is less than $\varepsilon$ for large enough $N$. This proves our claim that $h_s$ is well defined.

Next, we prove that $h_s$ is continuous in $[1, \infty)$. We initially leave out $t=1$; it is handled separately below.
\begin{prop}
Suppose that $s = s_0 s_1 s_2 \ldots \in \Sigma_K$. Then $h_s(t)$ is continuous as a function of $t \in (1,\infty)$.
\end{prop}
\begin{proof}
 Choose $\alpha$ with $0 < \alpha <1$ and let $q$ and $M$ be as specified in the previous proposition. Choose $T > q + 2M$ so that, if $\operatorname{Re} z> T$ and $|\operatorname{Im} z| < M_K$ (with $M_K$ defined as in Proposition \ref{anisotita}), then
\begin{equation}
|L_{s_i}^{'} (z)| < \alpha.\label{eq:p27}
\end{equation}
This is possible due to \eqref{eq:daveleq1} of Lemma \ref{davel}.
By Proposition \ref{anisotita}, if $t > T$, then
\begin{equation}
E^k(t) - M \leq \operatorname{Re} G_s^n (E^k(t)) \leq E^k(t) + M \label{eq:p16},
\end{equation}
for all $n, k \geq 0$.

We first prove the continuity of $h_s(t)$ for $t>T$.  Let $\varepsilon > 0$ and choose $k \in \fis$ so that $\alpha^k (3M+2 \pi ) < \varepsilon$. Given $t_0 > T$, choose $\delta$ such that, if $|t-t_0| < \delta$, then $|E^k(t) - E^k(t_0)| < M$. We claim that, if $|t-t_0| < \delta$, then 
$|h_s(t) - h_s(t_0)| < \varepsilon$. Indeed, we note that for such $t$ and each $n \geq 0$ we have
\begin{align*}
|G_{\sigma^k(s)}^n (E^k(t)) - G_{\sigma^k(s)}^n (E^k(t_0))| < 3M + 2 \pi.
\end{align*}
This follows since, by (\ref{eq:p16}) and our choice of $\delta$,
\begin{align*}
\left| \operatorname {Re} G_{\sigma^k(s)}^n ( E^k(t)) - \operatorname {Re} G_{\sigma^k(s)}^n ( E^k(t_0)) \right| < \left| E^k(t) - E^k(t_0) \right| + 2M < 3M 
\end{align*}
and
\begin{align*}
\left| \operatorname{Im} G_{\sigma^k(s)}^n ( E^k(t)) - \operatorname{Im} G_{\sigma^k(s)}^n ( E^k(t_0)) \right| < 2 \pi.
\end{align*}
Consequently, by \eqref{eq:p27} and \eqref{eq:p16} for $|t - t_0 | < \delta$ and $n \geq 0$,
\begin{align*}
\left| G_s^{n+k} (t) - G_s^{n+k}(t_0) \right| &= 
\left| L_s^k \circ G_{\sigma^k(s)}^n (E^k(t)) - L_s^k \circ G_{\sigma^k(s)}^n (E^k(t_0)) \right| 
\\&\leq \alpha^k \left| G_{\sigma^k(s)}^n (E^k(t)) - G_{\sigma^k(s)}^n (E^k(t_0)) \right|
\\&\leq \alpha^k (3M+2 \pi) < \varepsilon,
\end{align*}
from which it follows that $t \mapsto h_s (t)$ is continuous for any $s = s_0 s_1 s_2 \ldots \in \Sigma_K$ and $t > T$. 

We will now prove continuity for $1 < t \leq T$. If $1< t < T$, then there exists $k$ (depending on $t$) such that $E^k(t) > T$. Then, by the earlier part of the proof,
\begin{align*}
t \mapsto L_s^k \circ h_{\sigma^k(s)} (E^k(t))
\end{align*}
is continuous, since each inverse function of $f$ is well defined and continuous on the half-strips $H_{m,c}$; see the remark following the proof of Corollary \ref{corinv}. But this map is given by
\begin{align*}
t \mapsto \lim_{n \to \infty} L_s^k \circ G_{\sigma^k(s)}^n \circ E^k(t) = h_s (t),
\end{align*}
and since $k$ depends only on $t$, the result follows.
\end{proof}
We now prove continuity for $t=1$.

\begin{prop}
Suppose that $s \in \Sigma_K$. Then $h_s(t)$ is continuous at $t=1$.
\end{prop}
\begin{proof}
From Section \ref{sec:trapeziums}, we know that, for all $i \geq 0$, $L_{s_i}(z)$ maps $T^K$ (defined in that section as the union of the relevant trapeziums) strictly inside itself for any $c$ large enough, with $x=c$ being the line on which the right-hand sides of each of the trapeziums in $T^K$ lie. As we previously saw in Lemma~\ref{lj}, $L_{s_i}$ is a strict contraction with respect to the Poincaré  metric on $T^K$. We need to use this fact to prove our result, but also we want to benefit from the inequalities of Proposition \ref{anisotita}.

 To that end, note that for a value of $t$ that is sufficiently close to $1$, there exists some integer $N$ (dependent ont $t$) such that $E^N(t)$ is larger than the $q$ that is specified in Proposition \ref{anisotita}. Thus, for any $n > N$, we have $G_{\sigma^N ( s)}^n (E^N(t)) \in T^K$ for $c$ sufficiently large. We can now use the Poincaré metric to show that the distance between $h_s(1) = z(s)$ and 
\begin{align*}
h_s(t) = \lim_{N \to \infty} L_s^N \circ G_{\sigma^n(s)}^n (E^N(t))
\end{align*}
can be made arbitrarily small as $t \to 1^{+}$. 

To prove this, first we note that the endpoint $z(s)$ lies in $T_{s_0,c}$ for any sufficiently large $c>0$. Choose $c>0$ to be large enough so that 
\begin{enumerate} \item $f(T_{j,c})$ contains $T^K_c$ for all $1 \leq |j| \leq K$, and
\item all endpoints for the hairs corresponding to itineraries $s=s_0 s_1 s_2 \ldots$ with $|s_j| \leq K$ lie in $T^K_c$. 
\end{enumerate}
Then, $z(s)$ lies in $T_{j,c}$. Since $z(\sigma^k(s))$ lies in $T_{{s_k},c}$, we claim that $z(s)$ arises as the limit of successive preimages under $f$ of the trapeziums (in accordance with the itinerary).

 To prove this claim, we use the Branner--Hubbard criterion (see, for example, \cite[p. 233, Problem 2.5]{milnor}). Let us denote by $\operatorname{mod} A$ the conformal modulus of an open topological annulus $A \subset \mig$. The Branner--Hubbard criterion states that if $K_1 \supset K_2 \supset K_3 \supset \ldots$ is a nested sequence of compact connected subsets of $\mig$, with each $K_{n+1}$ contained in the interior of $K_n$, and if, further, each interior $K^{\circ}_n$ is simply connected (making each difference $A_n = K^{\circ}_n \setminus K_{n+1}$ a topological annulus), then if $\sum_{n=1}^{\infty} \operatorname{mod} A_n = \infty$, the intersection $\cap K_n$ reduces to a single point. The sets $A_n$ in our case, are the difference of the starting set with the corresponding preimage under $f$; that is, $A_1 = T_{s_0,c} \setminus \overline{L_{s_0}  (T_{s_1,c})}$ and $A_n = L_{s_0} \circ \ldots \circ L_{s_{n-2}} (T_{s_{n-1},c}) \setminus \overline{L_{s_0} \circ \ldots \circ L_{s_{n-1}}  (T_{s_{n},c})}$ for $n \geq 2$. But $f$ maps conformally between these trapeziums (since $f$ is entire on $\mig$ and the zeros of $f'$ lie on the lines $\cup_{k=0,\ldots,n-1} V_k$ as shown in Theorem \ref{cpsss}), making the conformal moduli in each step constant and thus proving our claim.
\end{proof}

Finally, we prove that to $z(s)$, for each $s \in \Sigma_K$, there corresponds a \emph{unique} curve that is attached to it and is parametrised by $t \mapsto h_s(t)$.

\begin{theorem}\label{hair1}
Let $s = s_0 s_1 s_2 \ldots \in \Sigma_K$. There is a unique hair attached to $z(s)$ and $t \mapsto h_s(t)$ is a parametrisation of this hair. In particular, this hair lies entirely in $R(s_0)$.
\end{theorem}
\begin{proof}
We first verify that $h_s$ is indeed a hair, following Definition \ref{hairdef}. We claim that $h_s(t)$ has itinerary $s$ for $t \geq 1$. Note that, since $f \circ L_{s_0}$ is the identity, we have, for $t \geq 1$, 
\begin{align*}
f \circ h_s(t) = \lim_{n \to \infty}  f \circ G_s^n (t) = \lim_{n \to \infty} G_{\sigma(s)}^{n-1} (E(t)) = h_{\sigma(s)} (E(t)).
\end{align*}
It follows that, for $t \geq 1$,
\begin{equation}
f^n \circ h_s(t) = h_{\sigma^n(s)}( E^n(t)). \label{eq:ela}
\end{equation}
Hence $f^n( h_s(t)) \in R (s_n)$ (which denotes the horizontal $2\pi$-width strip that corresponds to $s_n$) as required. Also, from Proposition \ref{anisotita} and \eqref{eq:ela},
\begin{align*}
E^n(t) - M \leq \operatorname{Re} f^n \circ h_s(t) \leq E^n(t) + M,
\end{align*}
for $n$ sufficiently large, where $M$ is as specified in Proposition \ref{anisotita}. Therefore, $\operatorname{Re} f^n  \circ h_s(t) \to \infty$ as $n \to \infty$ when $t>1$. Finally, since $t-M \leq \operatorname{Re} h_s(t) \leq t+M$ for $t> q$, it follows that $\operatorname{Re} h_s(t) \to \infty$ as $t \to \infty$. This proves that $h_s$ parametrises a hair. We will now show that this hair is unique.

Suppose that $h_s$ is not unique. Then there are at least two hairs attached to $z(s)$; consider two of them. We examine the following two cases.
\begin{itemize}
\item Suppose that the hairs meet in only a bounded set of points. Consider the last point of intersection; suppose that point is $\zeta$. Let $U$ be the unbounded open set consisting of the set of points contained in $R(s_0)$ that is bounded by the two hairs, has $\zeta$ in its boundary and can only access infinity from the right half-plane. We claim that the images of $U$ under $f^n$ are contained within the images of the hairs attached to $f^n(\zeta)$ (so, in $T_0(\nu)$) and therefore, by Montel's theorem, $U$ has to be in the Fatou set of $f$, which contradicts Lemma \ref{g}.

 To prove the claim, we consider a point $z \in U$ and intersect $U$ with the half-plane $\{w \in \mig : \operatorname{Re} w < \lambda \}$, with $\lambda > \operatorname{Re} z$; name the new bounded set $U_{\lambda}$. The crosscuts of $U_{\lambda}$ on the vertical $\{w \in \mig : \operatorname{Re} w = \lambda \}$ map under $\exp$ to arcs of a circle around the origin of radius $e^{ \lambda}$. When $\lambda$ is large enough, $f$ will map the crosscuts to a thin annulus around that circle. The set $f(U_{\lambda})$ can then be one of two bounded sets defined by $f(\partial U_{\lambda})$ inside some circle around the origin. But, since the two hairs have to lie in $R(s_1)$ following the itinerary $s$ and $f(U_{\lambda})$ does not contain the origin, $f (U_{\lambda})$ is a bounded region that is defined between them and has to lie in $R(s_1)$ as well. Its further forward images will lie in their respective strips in $T_0( \nu)$, thus avoiding the left half-plane. Since $\lambda > \operatorname{Re} z$ was arbitrary, the forward images under $f$ of the unbounded region $U$  also have to also lie in $T_0(\nu)$, thus proving our claim.
\item Suppose that the curves meet in an unbounded set of points. Since the hairs are unbounded closed sets that are not identical, there must exist a domain $U$ lying in $R(s_0)$ whose boundary lies entirely in the two hairs. As in the previous case, the images of $U$ under $f^n$ are bounded by the images of the hairs attached to $f^n(\zeta)$ and therefore, by Montel's theorem, $U$ is in the Fatou set of $f$. This contradicts Lemma $\ref{g}$. \qedhere
\end{itemize}
\end{proof}

We can now deduce our main result.
\begin{proof}[Proof of Theorem 1.2]
Let $f \in \mathcal{F}$. Theorem \ref{hair1} gives the existence of a Cantor bouquet in $T_0(\nu) \cup Q_0 \cup Q_1$. Consider an arbitrary hair of the Cantor bouquet, parametrised by $t \mapsto h_s(t)$. From Lemma~\ref{jcor}, $z(s)$ is in $J(f)$. Let $z = h_s(t)$ with $t>1$. Then $f^n(z) \to \infty$ in the angle $T_0(\nu)$, so there exists $N \in \fis$ such that for $n \geq N$, $f^n (z) \in T_0(\nu)$. Therefore, from Lemma \ref{g}, $z \in J(f) \cap A(f)$. 
\end{proof}

Finally, we have the following corollary of Theorem \ref{hair1} and Lemma \ref{curveswd}, which reveals a key part of the structure of the Julia set given by Theorem \ref{hair1}.

\begin{cor}
Let $f \in \mathcal{F}$. For all large enough $k \in \fis$, there exist two simple unbounded curves $\gamma_k$ and $\gamma_{-k}$ in $J(f)$, with their endpoints in $Q_0$ and $Q_1$ respectively, that lie entirely inside the strips $R(k)$ and $R(-k)$ respectively, and tend to infinity through $T_0(\nu)$.
\end{cor}

We note that the symmetry properties of the function $f$ allow us to extend the result of this corollary to $T_k (\nu)$, for $k=1, \ldots, p-1$.


\newpage
\begin{thebibliography}{9}


\bibitem{aarts} 
J. M. Aarts and L. G. Oversteegen, The geometry of Julia sets, 
\textit{Trans. Amer. Math. Soc.}, 338 (1993), 897--918.


\bibitem{baker2} 
I. N. Baker, Repulsive fixpoints of entire functions,
\textit{Math. Zeitschrift}, 104 (1968), 252--256.

\bibitem{baker} 
I. N. Baker, Wandering domains in the iteration of entire functions,
\textit{Proc. London Math. Soc.}, 49 (1984), 563--576.


\bibitem{lx}
K. Barański, X. Jarque and L. Rempe, Brushing the hairs of transcendental entire functions,
\textit{Topology Appl.}, 159 (2012), no. 8, 2102--2114

\bibitem{berg}
W. Bergweiler and B. Karpińska, On the Hausdorff dimension of the Julia set of a regularly growing entire function,
\textit{Math. Proc. Cambridge Philos. Soc.}, 148, 3 (2010), 531--551.


\bibitem{blanchard} 
P. Blanchard, Complex analytic dynamics on the Riemann sphere,
\textit{Bull. Amer. Math. Soc.}, II (1984), no. 1, 85--141.

\bibitem{dev3} 
C. Bodelón, R. L. Devaney, M. Hayes, G. Roberts, L. R. Goldberg and J. H. Hubbard, Hairs for the complex exponential family,
\textit{Internal. J. Bifur. Chaos Appl.}, 9, no. 8 (1999), 1517-1534.


\bibitem{dyn}
M. Brin and G. Stuck, \textit{Introduction to Dynamical Systems}, Cambridge University Press, 2002.

\bibitem{dev1} 
R. L. Devaney and F. Tangerman, Dynamics of entire functions near the essential singularity,
\textit{Ergod. Th. \& Dynam. Sys.}, 6 (1986), 489--503.


\bibitem{dev2} 
R. L. Devaney and M. Krych, Dynamics of Exp(z),
\textit{Ergod. Th. \& Dynam. Sys.}, 4 (1984), 35--52.

\bibitem{milnor}
J. Milnor, \textit{Dynamics in One Complex Variable},  Princeton University Press, 2006.

\bibitem{john}
J. Osborne, Spiders’ webs and locally connected Julia sets of transcendental entire functions,
\textit{Ergod. Th. \& Dynam. Sys.}, 33.4 (2013), 1146--1161.

\bibitem{polya}
G. Pólya and G. Szegő, \textit{Problems and Theorems in Analysis II}, Springer, 1976

\bibitem{rs1} 
P. J. Rippon and G. M. Stallard, Fast escaping points of entire functions,
\textit{Proc. London Math. Soc.}, 105 (2012), 787–820.
 

 \bibitem{rs3} 
P. J. Rippon and G. M. Stallard, On sets where iterates of a meromorphic function zip towards infinity,
\textit{Bull. London Math. Soc.}, 32 (2000), 528--536. 

\bibitem{sz}
D. Schleicher and J. Zimmer, Escaping points of exponential maps,
\textit{J. London Math. Soc.}, 67 (2003), 380--400.


\bibitem{dave2} 
D. J. Sixsmith, Julia and escaping set spiders' webs of positive area,
\textit{International Mathematics Research Notices}, 19 (2015), 9751--9774.
\bibitem{tit1} 
E. C. Titchmarsh, \textit{The theory of functions},  Oxford University Press, 1939.
\end{thebibliography}
\end{document}